\documentclass[sn-mathphys, Numbered]{sn-jnl}
\newcommand{\orcidID}[1]{\href{https://orcid.org/#1}{{\textsf{\tiny ORCID}}}}

\usepackage{float}
\usepackage{mathtools}   
\usepackage{xcolor}
\usepackage{amsmath,amssymb,amsfonts}
\usepackage{amsthm}
\usepackage{natbib}
\usepackage{mathtools}
\usepackage{chngcntr}

\newcommand{\R}{\mathbb{R}}
\newcommand{\eps}{\varepsilon}
\newcommand{\Om}{\Omega}
\newcommand{\Omzero}{\Omega_{0}}
\newcommand{\Phit}{\Phi}
\newcommand{\DFE}{\overset{\circ}{x}}
\newcommand{\EE}[1][\eps]{x^{*}_{#1}}
\newcommand{\BRN}{\mathcal{R}_{0}}
\newcommand{\epstilde}{\tilde\eps}

\newcommand{\Ws}{W^{s}}

\newcommand{\real}{\mathbb{R}}
\newcommand{\naturals}{\mathbb{N}}
\newcommand{\xfast}{{\bar x}}
\newcommand{\rfast}{{\bar R}}
\newcommand{\cfast}{{\bar C}}
\newcommand{\epsinv}[1][1]{\eps^{-#1}}
\newcommand{\critman}{\mathcal{M}_\infty}
\newcommand{\critsubman}[1][\infty]{\mathcal{K}_{#1}}
\newcommand{\infEE}[1][\varepsilon]{I^*_{#1}}
\newcommand{\EEfast}[1][\eps]{\bar{x}^*_{#1}}
\newcommand{\ffast}{f_{\text{fast}}}
\newcommand{\fslow}{f_{\text{slow}}}
\newcommand{\bigo}{\mathcal{O}}
\newcommand{\poly}{{\mathcal{P}}}

\theoremstyle{thmstyleone}
\newtheorem{thm}{Theorem}[section]
\newtheorem{lemma}[thm]{Lemma}
\newtheorem{cor}[thm]{Corollary}
\newtheorem{prop}[thm]{Proposition}
\newtheorem{rem}[thm]{Remark}

\theoremstyle{thmstylethree}
\newtheorem{dfn}[thm]{Definition}

\begin{document}

\title{Compartmental epidemiological models with infection-driven immune escape}

\author*[1]{\fnm{Daniela R.} \sur{Ram\'{i}rez-Guti\'{e}rrez}\orcidID{0009-0008-1550-8022}}\email{danielarenee.ram@gmail.com}

\author[2]{\fnm{Lajos}  \sur{L\'oczi} \orcidID{0000-0002-7999-5658}}\email{lajos.loczi@kaust.edu.sa}

\author[2]{\fnm{Carlos} \sur{Mu\~{n}oz-Moncayo}
\orcidID{0000-0003-0982-0817}}\email{carlos.munozmoncayo@kaust.edu.sa}

\author[2]{\fnm{David I.} \sur{Ketcheson}\orcidID{0000-0002-1212-126X}}\email{david.ketcheson@kaust.edu.sa}

\affil*[1]{\orgdiv{Departamento de Actuar\'{\i}a, F\'{\i}sica y Matem\'aticas}, \orgname{Universidad de las Am\'ericas Puebla (UDLAP)}, \orgaddress{\city{Puebla}, \postcode{72810}, \country{Mexico}}}

\affil[2]{\orgdiv{Computer Electrical and Mathematical Science and Engineering Division (CEMSE)}, \orgname{King Abdullah University of Science and Technology (KAUST)}, \orgaddress{\city{Thuwal}, \postcode{23955-6900}, \country{Saudi Arabia}}}

\abstract{We propose and study a compartmental epidemiological model that incorporates both waning immunity and mutation-driven immune escape into a single-strain framework. The model, termed SIRCm, extends the SIRC (Susceptible--Infected--Recovered--Cross-immune) model, in which recovered individuals pass through an intermediate cross-immune class before returning to full susceptibility. In SIRCm, both the rate of immune waning and the rate of immune escape are amplified by a feedback function tied to the current level of infection. We study two formulations of this feedback, a transmission-driven model, in which mutation opportunities arise at the point of new infection, and a prevalence-driven model, in which they arise in proportion to the infected population alone. For both models and any feedback strength, we prove existence of an endemic equilibrium using persistence theory, and characterize its behavior analytically in the weak and strong-feedback limits, recovering the SIRC and SIS models, respectively, as limiting cases. Unlike the SIRC model, we show that the SIRCm endemic equilibrium is unstable in certain parameter regimes. This leads to novel behavior, even in the absence of seasonal forcing.}


\maketitle


\newpage
\section{Background and motivation}
The classical SIR (Susceptible--Infected--Recovered) model and its variants
have proven surprisingly effective at describing the spread of infectious
diseases at the population level. A central assumption in the basic SIR
framework is that recovered individuals acquire permanent immunity, so that
each individual is infected at most once. In practice, two distinct mechanisms
can invalidate this assumption: (i) waning of immune memory over time, and (ii)
viral mutation resulting in immune escape, whereby the pathogen evolves to
circumvent the defenses of previously infected hosts.

The first mechanism---immunity waning---can be incorporated rather simply into
SIRS-type models, in which recovered individuals return to the
susceptible class at some fixed rate \citep{casagrandi2006, pease1987}.
The second mechanism---mutation-driven immune escape---is potentially more complex and can be modeled in various ways.
A well-established approach is to
explicitly model multiple strains; see, e.g.,  \cite{andreasen1997dynamics, andreasen2003dynamics, gog2002dynamics, ferguson2003ecological}.
On the other hand, if one assumes that mutation happens
at a steady rate over time, then the resulting immune escape effect can be incorporated
in the same way as waning, in a single-strain model.

In many epidemic phases, infections may be dominated by one antigenic strain or lineage, with newly emerging strains subsequently replacing previously dominant ones.
Furthermore, in practice mutations do not occur at a
fixed rate; rather, mutation opportunities arise proportionally to the
number of active infections, since each replication event inside a host is a
potential mutation event \citep{boni2004influenza, boni2006, zhang2022}. 
The COVID-19 pandemic has made this interaction especially visible.
Successive variants of concern---Alpha, Delta, Omicron, and its
subvariants---demonstrated that immune escape could play a major role in
driving renewed epidemic waves \citep{markov2023,balloux2022}.
The pandemic also generated an unprecedented volume of genomic sequence data,
providing new opportunities to calibrate and validate mathematical models of
viral evolution at the population level \citep{saha2026}.
These observations motivate an intermediate modeling approach: rather than resolving individual strains, we ask whether the population-level effect of mutation can be represented through a state-dependent rate of immune escape.

In the present work, we formulate and study an extension of the SIRC model \citep{casagrandi2006,li2017} that accounts for the relationship between prevalence and mutation rate, yet avoids the complexity of explicitly tracking multiple strains.
The influence of mutation is implicit and appears solely through an acceleration of individuals transferring to successively less-immune groups.
Our model includes additional terms that provide a coarse-grained representation of mutation-driven immune escape,
without resolving the coexistence and replacement of individual strains explicitly.
Because this model is intended to incorporate the effect of viral \emph{mutation} (distinct from that of the host's waning immunity) we refer to it as SIRCm.
Unlike the SIRC model, but in agreement with some multi-strain models (see, e.g., \cite{andreasen1997dynamics, andreasen2003dynamics},
SIRCm can exhibit non-steady long-term behavior even in the absence of seasonal forcing.
Our approach shows that important qualitative consequences of mutation-driven immune escape can arise within a low-dimensional single-strain model.

An additional feature of SIRCm is that it
smoothly connects the SIRC model with a traditional SIS-type model.
A key parameter in the models describes how strongly new infections feed back into immune escape.
In the weak-feedback limit, the model reduces to SIRC, while in the strong-feedback limit it approaches an SIS-type model.
As a result, the endemic equilibrium undergoes a rapid change around a critical intermediate feedback strength. 

Our main contributions are the following.
\begin{itemize}
    \item In Section \ref{sec:model}, we propose the SIRCm model, a single-strain model of infection-driven immune escape,
    and study two natural choices of feedback: a \emph{transmission-driven} model, in which mutation
    opportunities arise at transmission events, and a \emph{prevalence-driven} model, in which they arise in proportion to the standing infected population.
    \item In Section \ref{sec: persistence theory}, via persistence theory, we
    prove the existence of an endemic equilibrium for arbitrary
    feedback strength, together with uniform strong persistence of the infection.
    \item In Sections \ref{sec:weak_feedback} and \ref{sec:strong_feedback},
    we prove that 
    the endemic equilibrium is unique and locally asymptotically stable
    for weak and strong enough feedback,
    and we provide a recursive approach to computing asymptotic
    approximations---with increasing order of accuracy---of the endemic equilibrium in
    both regimes.
    For strong feedback, a fast--slow decomposition identifies an
    attracting, locally invariant manifold in which the slow dynamics 
    are those of the classical SIS model.
    \item In Section \ref{sec:moderate feedback}, we show, by numerical continuation of the Hopf
    bifurcation set 
    with respect to the feedback strength,
    that the endemic equilibrium of the transmission-driven model  loses stability inside
    a bounded window of feedback strengths.
    Inside that window, trajectories approach a limit cycle and
    prevalence oscillates continuously.
    These oscillations provide a mechanism
    for recurrent epidemic waves without relying on seasonal forcing.
    \item In Section \ref{sec:seasonally_forced}, we examine the SIRCm model under seasonal forcing, showing that mutation feedback can alter the resulting dynamics (e.g., annual cycles, bi-annual cycles, chaotic behavior), relative to the seasonally forced SIRC model. Further examples are given in Appendix \ref{app:seasonal}. 
    \item Finally, Section \ref{sec:discussion} discusses the biological implications of these findings. 
\end{itemize}

The supporting derivations for the equilibrium polynomial of the transmission-driven model and the stability analysis of the prevalence-driven model are collected in Appendices \ref{app:quartic} and \ref{appendixB}, respectively.

\section{Model formulation}\label{sec:model}
We propose and study the following system of ordinary differential equations (ODEs), henceforth referred to as the \textit{SIRCm model}:
\begin{subequations}
\label{eq:SIRCm}
\begin{align}
    \dot S &= \mu(1-S) - \beta SI + (1+ \eps_2g(x))\gamma C, \label{eq:SIRCm S}\\
    \dot I &= \beta SI + \sigma\beta CI - (\mu+\alpha)I,\label{eq:SIRCm I} \\
    \dot R &= (1-\sigma)\beta CI + \alpha I - \mu R - (1+ \eps_1g(x))\delta R, \label{eq:SIRCm R}\\
    \dot C &= (1+ \eps_1g(x))\delta R - \beta CI - \mu C - (1+ \eps_2g(x))\gamma C, \label{eq:SIRCm C}
\end{align}
\end{subequations}
where
\[
  \mu, \beta, \alpha, \delta, \gamma > 0, \qquad \sigma \in [0,1], \qquad \eps_1, \eps_2 \geq 0.
\]
The solution vector of  \eqref{eq:SIRCm}---with suitable initial conditions $(S(0),I(0),R(0),C(0))$---is denoted by \[x(t)\coloneqq(S(t),I(t),R(t),C(t)).\]
The dependent variables correspond to the same compartments as in the SIRC model \citep{casagrandi2006}: \textit{susceptible} ($S$), \textit{infectious} ($I$), \textit{recovered} ($R$), and \textit{cross-immune} ($C$).  Each represents a fraction of the total population, and takes values in $[0,1]$.
As for the parameters, $\mu$ is the \textit{natural birth} (or mortality) \textit{rate}, assumed equal across compartments so that the total population is conserved; 
$\beta$ is the \textit{transmission rate}; and $\alpha$ is the \textit{recovery rate} from active infection;
$\delta$ is the rate at which \textit{complete immunity wanes} into the cross-immune class $C$ ($R\to C$); $\gamma$ is the rate at which \textit{cross-immunity} \textit{wanes} into full susceptibility ($C\to S$);
$\sigma$ is the fraction of $C$–$I$ contacts  that result in \textit{renewed active infection} rather than successful clearance. Finally, $\eps_1$ and $\eps_2$ are the \textit{mutation feedback parameters}, scaling how strongly the current level of active infection amplifies immune-escape $\delta$ and waning $\gamma$, respectively, via the terms $(1+ \eps_1g(x))$ and $(1+ \eps_2g(x))$ multiplying each. Setting $\eps_1 =\eps_2 =0$ recovers the SIRC model of \cite{casagrandi2006}. 
A flowchart of the model dynamics is shown in
Figure \ref{fig:flowdiagram}. 

The effect of mutation is modeled implicitly through the terms involving the function $g(x)$.
We consider two choices of $g$, differing in how the mutation feedback is driven by the state of the system: 
\begin{itemize}
\item the \textit{transmission-driven model}, in which mutation opportunities arise at transmission events, given by

\begin{equation} \label{t-dv_formula}
    g(x) = SI.
\end{equation}

\item the \textit{prevalence-driven model}, in which mutation opportunities arise from the currently infected population, independent of ongoing transmission, given by

\begin{equation} \label{p-dv_formula}
    g(x) = I.
\end{equation}

\end{itemize}
These variations are compared directly in Section~\ref{sec: stability and bifurcation}.

Realistic ranges for the parameters are shared with those given in Table 1 of \cite{casagrandi2006} for Influenza A. Unless otherwise noted, the base parameter values used throughout the numerical results of this paper follow those of Figure 2 of \cite{casagrandi2006}: $\mu=0.02$, $\alpha=365/3$, $\delta=1.6^{-1}$, $\sigma=0.07874$, $\gamma=0.35$, with $\beta_0$ and the mutation feedback parameters $\eps_1,\eps_2$ 
varied as indicated in each figure. 

\begin{figure}
    \centering
    \includegraphics[width=0.9\textwidth]{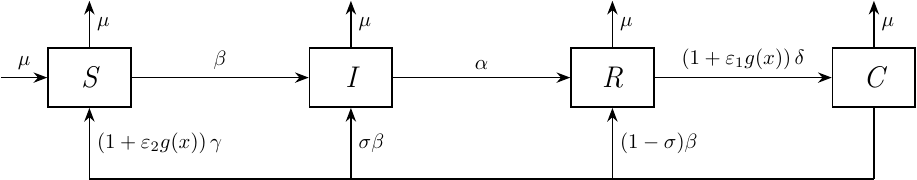}
    \caption{(see Section \ref{sec:model}) Flow diagram of the SIRCm model}
    \label{fig:flowdiagram}
\end{figure}

\subsection{Preliminaries}\label{sec:prelim}

The state space for the SIRCm model is the set
\[
  \Om \;\coloneqq\; \{ 
  (S,I,R,C) \in 
  [0,1]^4\;:\; S+I+R+C = 1\}
\]
equipped with the standard Euclidean metric.
We note its properties for future reference.
\begin{prop}\label{prop:state-space}
The state space $\Om$ is non-empty, compact, and convex.
\end{prop}

\begin{prop}\label{prop:forward-invariance}
    $\Om$ is forward invariant: for any initial condition $x(0)\in\Om$, the unique solution of \eqref{eq:SIRCm} 
    satisfies $x(t) \in \Om$ for all $t \geq 0$.
\end{prop}
\begin{proof}
    \emph{Conservation.} Let $P \coloneqq S+I+R+C$. Summing the four equations of \eqref{eq:SIRCm} 
    yields 
    $
      \dot P = \mu(1-P)
    $.
    The solution of this linear scalar ODE is $P(t) = 1 + (P(0)-1)e^{-\mu t}$. In particular, $P(0) = 1 \Longrightarrow P(t) = 1$ for all $t \geq 0$.\\
       \emph{Non-negativity.} We verify that the vector field is quasi-positive on  
       $[0,+\infty)^4$, i.e., $\dot x_{i} \geq 0$ whenever $x_{i} = 0$ ($1\le i\le 4$) and the remaining coordinates are non-negative.
    \begin{itemize}
      \item For $I$: $\dot I = I[\beta S + \sigma\beta C - (\mu+\alpha)]$ is linear in $I$, with solution
      \[
      I(t) = I(0)\,\exp\!\left(\int_{0}^{t} 
        [\beta S(u) + \sigma\beta C(u) - (\mu+\alpha)]\,du\right).
      \]
      Since the exponential factor is strictly positive, $I(0) \geq 0 \Longrightarrow I(t) \geq 0$ for all $t \geq 0$.      
      \item For $S$: $\dot S|_{S=0} = \mu + (1+\eps_2g(x))\gamma C > 0$, since $g(x)\geq 0$.
      \item For $R$: $\dot R|_{R=0} = (1-\sigma)\beta CI + \alpha I \geq 0$.
      \item For $C$: $\dot C|_{C=0} = (1+\eps_1g(x))\delta R \geq 0$, since $g(x)\geq 0$.
    \end{itemize}
    
    Nagumo's theorem \citep{blanchini1999} then implies $x(t) \ge 0$, and---together with the conservation property---$x(t) \in \Om$ for all $t \geq 0$ for which the solution of \eqref{eq:SIRCm} is defined. 
    
    The right-hand sides of \eqref{eq:SIRCm} are polynomials in $(S,I,R,C)$, hence they are locally Lipschitz and continuous functions, guaranteeing local existence and uniqueness of the corresponding initial-value problem. Finally, from the extension theorem for ODEs, we get that the solution $x(t)$ is defined for \textit{all} $t\ge 0$, since the solution does not leave the compact set $\Om$.
\end{proof}

We assume that mutation feedback is related to infections, with the following conditions.
\begin{dfn}\label{dfn:funcg}
Let $g$ be a $\mathcal{C}^\infty$ function on $\Om$,  satisfying:
\begin{enumerate}
    \item $g(x)$ depends exclusively on the state variables $S$ and $I$;
    \item $g\geq 0$ on $\Om$;
    \item $I=0 \implies g=0$;
    \item $g=0 \implies SI=0$.
\end{enumerate}
\end{dfn}

The \textit{basic reproduction number}, denoted by $\BRN$, is the expected number of secondary infections generated by a single infectious individual in an otherwise fully susceptible population.
The SIRCm model has the same basic reproduction number as the SIRC model:
\begin{equation}\label{R0definition}
\BRN \coloneqq \frac{\beta}{\mu+\alpha}.
\end{equation}
It also has the same
disease-free equilibrium as the SIRC model,
with the same stability properties.
Here and throughout this work,
$J(x)$ denotes the Jacobian matrix of the right-hand side of \eqref{eq:SIRCm}.

\begin{prop}\label{prop:DFE}
    SIRCm \eqref{eq:SIRCm} has a unique disease-free equilibrium (DFE) on 
    \begin{equation}\label{Omegazerodefinition}\Omzero\coloneqq \{x\in\Om : I=0\},\end{equation}
    given by
    $
      \DFE \coloneqq (1,0,0,0).
    $
    Setting $\BRN \coloneqq \beta/(\mu+\alpha)$, $\DFE$ is hyperbolic whenever $\BRN \neq 1$, locally asymptotically stable when $\BRN < 1$, and unstable when $\BRN > 1$. The eigenvalues of $J(\DFE)$  are
    \[
      \{-\alpha+\beta-\mu,\ -\delta(1+g(\DFE)\eps_1)-\mu,\ -\gamma(1+g(\DFE)\eps_2)-\mu,\ -\mu\}.
    \]  
    For $\BRN > 1$, the three stable eigendirections span the 3-dimensional subspace  $\{x\in\mathbb{R}^4: I=0\}$.
\end{prop}

\section{The existence of an endemic equilibrium via persistence theory}
\label{sec: persistence theory}
In this section, we show the existence of a second equilibrium in the interior of $\Om$ (denoted by $\mathrm{int}(\Om)$) whenever the basic reproduction number satisfies $\BRN > 1$.

By eliminating $S,R,C$ from the equilibrium equations of SIRCm \eqref{eq:SIRCm}, we get a polynomial equation $\poly(I,\eps_1, \eps_2)=0$, quartic in $I$; the explicit form of $\poly$ is given in Appendix \ref{app:quartic}. A root of $\poly(\cdot,\eps_1, \eps_2)$ in $(0,1]$ would yield an endemic equilibrium provided that $(S,I,R,C)\in \Om$.
However, by using standard algebraic techniques, we were unable to prove the existence of such an equilibrium in the simplex $\Om$, 
     even when we reduced the dimensionality of the state space to a 3-dimensional tetrahedron by eliminating one of the variables from the conservation property $S+I+R+C=1$.  

For this reason, we turn 
to the 
persistence-theoretic framework  
of \cite{SmithThieme}.     
Here, we denote the solution operator for \eqref{eq:SIRCm} by $\Phi(t, x_0) \coloneqq x(t)$, and denote the time-$t$ map by $\Phi_t(\cdot) \coloneqq \Phi(t, \cdot)$.
 
\begin{thm}\label{prop:endemic-equilibrium}
Assume the basic reproduction number satisfies $\BRN > 1$. Then, for any $\eps_{1},\eps_{2}\geq0$, SIRCm \eqref{eq:SIRCm} has an endemic equilibrium: there exists $x^{*} \in \Om$ with $\rho(x^{*}) = I^{*} > 0$ and $\Phit(t,x^{*}) = x^{*}$ for all $t \geq 0$.
\end{thm}

Our strategy for establishing Theorem \ref{prop:endemic-equilibrium} relies on verifying the hypotheses of the fixed-point theorem \citet[Theorem 6.2]{SmithThieme}. As a first step, we show that $\Phi$ generates a continuous semiflow on $\Omega$ and possesses a global compact attractor.
Let $\mathcal{A} \coloneqq \omega(\Omega)$ denote the $\omega$-limit set of $\Omega$. We have the following basic properties.

\begin{prop}\label{prop:semiflow}
     $\Phit : [0,+\infty) \times \Om \to \Om$ is a continuous semiflow on $\Om$ \citet[Definition 1.33]{SmithThieme}.
\end{prop} 

\begin{prop}\label{prop:compact-attractor}
    $\mathcal{A}$ is a non-empty, compact, invariant subset of $\Om$ that attracts every subset of $\Om$; that is, $\mathcal{A}$ is a compact attractor of neighborhoods of compact sets in $\Om$ \citet[Definition 2.22]{SmithThieme}.
\end{prop}
\begin{proof}
    Since $\Om$ is compact (Proposition~\ref{prop:state-space}) and forward invariant under $\Phit$ (Proposition~\ref{prop:forward-invariance}), every trajectory starting in $\Om$ remains in $\Om$, hence any sequence $\Phit(t_i, x_i)$ with $t_{i}\to\infty$ and $x_{i}\in\Om$ has a convergent subsequence. Therefore, $\Phit$ is asymptotically compact on $\Om$ \citet[Definition 2.9]{SmithThieme}.
    By  
    \citet[Proposition 2.10]{SmithThieme}, $\mathcal{A}$ 
    is non-empty, compact, 
    and attracts $\Om$. 
    Invariance follows from \citet[Theorem 2.11]{SmithThieme}, making $\mathcal{A}$ a compact attractor of neighborhoods of compact sets in $\Omega$ \citet[Definition 2.22]{SmithThieme}.
\end{proof}

To apply the persistence criteria of \cite{SmithThieme}, we measure the distance from disease extinction using the persistence function
$\rho : \Om \to [0,+\infty)$ defined by 
\begin{equation}\label{para:persistence-function}
    \rho(S,I,R,C) \coloneqq I,
\end{equation}
whose zero-set coincides with the extinction set $\Omega_0 = \rho^{-1}(\{0\})$ introduced in \eqref{Omegazerodefinition}. Note that $\rho$ is continuous and concave on $\Omega$.  

In Propositions \ref{prop:Omega0-invariant}--\ref{prop:DFE-attractor} below, we first show that all boundary trajectories in $\Omega_0$ inevitably collapse to the disease-free equilibrium 
$\DFE$, rendering the boundary dynamics acyclic.

\begin{prop}\label{prop:Omega0-invariant}
    The extinction set $\Omzero$ is forward invariant under the SIRCm \eqref{eq:SIRCm} flow.
\end{prop}

\begin{prop}\label{prop:DFE-attractor}
    Every trajectory in $\Omzero$ converges to the disease-free equilibrium $\DFE$. Consequently, $\{\DFE\}$ is 
    the global compact attractor of the boundary semiflow $\Phit|_{\Omzero}$, 
    and the unique non-empty compact invariant subset of $\Omzero$. 
\end{prop}
\begin{proof}
    Since $I=0 \implies g=0$ (Definition \ref{dfn:funcg}), \eqref{eq:SIRCm} reduces on $\Omzero$ to the linear triangular system
    \[
      \dot R = -(\mu+\delta)R, \qquad
      \dot C = \delta R - (\mu+\gamma)C, \qquad
      \dot S = \mu(1-S) + \gamma C,
    \]
    subject to $S+R+C=1$. Solving 
    sequentially yields $R(t) \to 0$ and $C(t) \to 0$ as $t \to \infty$, which forces $S(t) \to 1$ by conservation.  
    Hence every trajectory starting in $\Omzero$ converges to $\DFE= (1, 0, 0, 0)$.

    As a consequence, the union of $\omega$-limit sets of boundary trajectories is simply $\tilde{\Omega} \coloneqq \bigcup_{x \in \Omega_0} \omega(x) = \{\DFE\}$. In the persistence framework of \citet[Section 8.3]{SmithThieme}, $\tilde{\Omega}$ is covered by the single invariant set $M_1 = \{\DFE\}$ (taking $k=1$). The required acyclicity condition \citet[Definition 8.14]{SmithThieme}  
    holds trivially because the associated directed graph consists of a single vertex $M_1$ with no self-loops.

\end{proof}

Having characterized the boundary dynamics on $\Omega_0$, we next turn to the behavior near the disease-free equilibrium $\DFE$. To apply persistence theory when $\BRN > 1$, we must show that $\DFE$ acts as a local repeller for initial states with positive infection. This requires two steps: establishing in Lemma \ref{prop:isolated} that $\{\DFE\}$ is locally isolated as an invariant set, and proving in Lemma \ref{prop:repelling} that no trajectory starting off the boundary can converge to $\DFE$.

\begin{lemma}\label{prop:isolated}
    Assume $\BRN \neq 1$. Then $\{\DFE\}$ is isolated in $\Om$ \citet[Definition 5.14]{SmithThieme}; that is, there exists a 
    neighborhood $V$ of $\DFE$ in $\Om$ such that every non-empty compact invariant set  
    $K \subset V$ satisfies $K = \{\DFE\}$.
\end{lemma}
\begin{proof}
    Since $\BRN \neq 1$, every eigenvalue of the Jacobian $J(\DFE)$ has non-zero real part (Proposition~\ref{prop:DFE}), so $\DFE$ is a hyperbolic equilibrium. By the Hartman--Grobman theorem, $\Phit$ is topologically conjugate near $\DFE$ to the linearized flow $\dot y = J(\DFE)y$ near the origin. Since no eigenvalue of $J(\DFE)$ has zero real part,
    this linear flow has no non-empty compact invariant set in a small neighborhood of $0$ other than $\{0\}$ itself. 
    Topological conjugacy transfers this property to $\Phit$, 
    yielding a neighborhood $V$ of $\DFE$ whose only non-empty compact invariant subset is $\{\DFE\}$.
\end{proof}

\begin{lemma}\label{prop:repelling}
    Assume $\BRN > 1$. Then $\{\DFE\}$ is weakly $\rho$-repelling \citet[Definition 8.15]{SmithThieme}: there is no $x \in \Om$ with $\rho(x) = I > 0$ such that $\Phit(t,x) \to \DFE$ as $t\to\infty$.
\end{lemma}
\begin{proof}
    By Proposition~\ref{prop:DFE}, the stable eigenspace $E^{s}$ of $J(\DFE)$ 
    is 3-dimensional and given by $E^{s}=\{I=0\}$. 
    One can directly see (similarly to  Proposition~\ref{prop:Omega0-invariant}) that the subspace  $\{I=0\}$ containing $\DFE$ is invariant under the non-linear flow $\Phi$.  
    By the local stable manifold theorem, there exists a unique local stable manifold $W^s_{\text{loc}}(\DFE)$ through $\DFE$, tangent to $E^s = \{I = 0\}$ at $\DFE$.
    Uniqueness, tangency and invariance force $\Ws_{\mathrm{loc}}(\DFE) = \{I=0\}\cap V$ for some neighborhood $V$ of $\DFE$.
    
    Now, suppose there exists $x\in\Om$ with $\rho(x)>0$---that is, $I(0) > 0$ for this trajectory---such that $\Phit(t,x)\to\DFE$ as $t\to\infty$. Then the 
      trajectory  must eventually enter the neighborhood $V$, and because it converges to $\DFE$, it must lie on 
      $W^s_{\mathrm{loc}}(\DFE)\subset\{I=0\}$ for all sufficiently large $t$. This would mean $I(t_1)=0$ for some finite $t_1>0$. 
      However, because $I(t) \equiv 0$ is a constant solution to the infected subsystem, reaching $I(t_1) = 0$ in finite time from $I(0) > 0$ violates the uniqueness of solutions to the initial value problem (see the proof of Proposition~\ref{prop:forward-invariance}).      
Therefore, no such $x$ can exist.
\end{proof}

We can now elevate local instability into global non-extinction across $\Omega$.

\begin{prop}\label{prop:uniform-weak-persistence}
Assume $\BRN > 1$. Then $\Phit$ is uniformly weakly $\rho$-persistent: 
there exists $\eta_{0} > 0$ such that
\[
  \limsup_{t\to\infty} I(t) > \eta_{0} 
  \quad \text{whenever } I(0) > 0.
\]
\end{prop}

\begin{proof}
The standing hypothesis~(H) of \citet[Section~8.3]{SmithThieme} holds with $B=\Om$: since $\Om$ is compact and forward invariant (Proposition~\ref{prop:forward-invariance}), every trajectory trivially approaches $\Om$, and $\Om \cap \{\rho \leq c\}$ is compact for every $c>0$.

We now apply \citet[Theorem 8.17]{SmithThieme} taking $k=1$ and $M_1 = \{\DFE\}$. 
The set $M_{1}$ is compact and invariant (our Proposition~\ref{prop:DFE-attractor}), isolated in $\Om$ (Lemma~\ref{prop:isolated}), weakly $\rho$-repelling (Lemma~\ref{prop:repelling}), and acyclic as a single set. The theorem yields uniform weak $\rho$-persistence of $\Phit$.
\end{proof}

We sharpen this result to uniform strong persistence 
by applying Theorem 4.5 of \cite{SmithThieme} with $B=\Om$.

\begin{cor}\label{cor:uniform-strong-persistence}
Assume $\BRN > 1$. Then $\Phit$ is uniformly (strongly) $\rho$-persistent: there exists $\eta_0 > 0$ such that
\[
  \liminf_{t\to\infty} I(t) > \eta_{0} 
  \quad \text{whenever } I(0) > 0.
\]
\end{cor}

\begin{rem}
 \citet[Theorem 6.2]{SmithThieme}---
which we utilize below to establish the existence of the endemic equilibrium---requires only uniform \emph{weak} persistence. However, Corollary~\ref{cor:uniform-strong-persistence} strengthens the epidemiological interpretation,
proving that, for $\BRN > 1$, any trajectory starting with an initial infection remains strictly bounded away from disease extinction in the long run.
\end{rem}

Uniform persistence implies that trajectories originating off the boundary do not accumulate on $\Omega_0$. 
This guarantees the existence of a compact persistence attractor bounded strictly away from the extinction boundary $\Omega_0$.

\begin{prop}\label{prop:persistence-attractor}
Assume $\BRN > 1$. The set
\[
  \mathcal{A}_{1} \;\coloneqq\; \overline{\bigcup_{x \in \Om \setminus \Omzero} \omega(x)},
\]
 where $\overline{(\cdot)}$ denotes the closure operation, is a non-empty, compact, invariant subset of $\Om \setminus \Omzero$
 that is uniformly $\rho$-positive; that is, there exists $\eta_1 > 0$ such that $\rho(x) \ge \eta_1$ for all $x \in \mathcal{A}_1$.
\end{prop}
\begin{proof}
Apply \citet[Theorem 5.7]{SmithThieme}: since $\Om$ is compact (Proposition~\ref{prop:state-space}) and $\Phit$ is uniformly weakly $\rho$-persistent (Proposition~\ref{prop:uniform-weak-persistence}), the set $\mathcal{A}_{1}$ 
possesses all stated properties.
\end{proof}

We now combine all ingredients—the global attractor $\mathcal{A}$, the persistence attractor $\mathcal{A}_1$, the persistence function $\rho$, and boundary invariance—to apply the persistence-theoretic fixed-point theorem.

\begin{proof} (of Theorem \ref{prop:endemic-equilibrium}).
    We verify the hypotheses of \citet[Theorem 6.2]{SmithThieme}. The set $\Om$ is a closed and convex subset of the Banach space $\R^{4}$ (Proposition~\ref{prop:state-space}); the persistence function $\rho$ is continuous and concave (\ref{para:persistence-function}); $\Phit$ is a continuous semiflow (Proposition~\ref{prop:semiflow}); $\Phit$ has a compact attractor $\mathcal{A}$ 
    of neighborhoods of compact sets (Proposition~\ref{prop:compact-attractor}), and 
     has a compact persistence attractor $\mathcal{A}_1 \subset \Omega \setminus \Omega_0$ 
    (Proposition~\ref{prop:persistence-attractor}); $\Phit_{t}$ 
    is conditionally compact (hence condensing) for all $t > 0$ in finite dimension $\mathbb{R}^4$, 
    and $\Phit$ is uniformly weakly $\rho$-persistent (Proposition~\ref{prop:uniform-weak-persistence}).
    
    It remains to verify hypothesis~(H1): that no total trajectory 
    $\phi \colon \mathbb{R} \to \mathcal{A}$
    satisfies $\rho(\phi(-r)) > 0$ and $\rho(\phi(s)) > 0$ for some $r, s \ge 0$ while $\rho(\phi(0)) = 0$. 
    Indeed, if $\rho(\phi(0)) = 0$, then $\phi(0) \in \Omega_0$.
    By forward invariance of $\Omzero$ (Proposition~\ref{prop:Omega0-invariant}), 
    $\phi(s) \in \Omega_0$ for all $s \ge 0$, which implies $\rho(\phi(s)) = 0$ for all $s \ge 0$. Thus, no trajectory can leave $\{I > 0\}$, hit the extinction boundary $\{I = 0\}$, and return to $\{I > 0\}$. Hypothesis (H1) holds.
    
    By \citet[Theorem 6.2]{SmithThieme},
    $\Phi$ has a fixed point $x^* \in \Omega$ with $\rho(x^*) = I^* > 0$. Therefore, $x^*$ is an endemic equilibrium of SIRCm \eqref{eq:SIRCm}, satisfying $\Phi(t, x^*) = x^*$ for all $t \ge 0$.
\end{proof}

\begin{rem}
We emphasize that the existence and stability of the disease-free equilibrium (Proposition~\ref{prop:DFE}) and the persistence-theoretic existence of an endemic equilibrium (Theorem~\ref{prop:endemic-equilibrium}) are established for general $g$ satisfying Definition~\ref{dfn:funcg}. These results continue to hold if we replace the expressions $(1 + \varepsilon_1 g(x))$ and $(1 + \varepsilon_2 g(x))$ by two $C^1$-smooth functions $h_j : [0, \infty)^2 \times [0,1]^2 \to [1, \infty)$ depending on $(\varepsilon_1, \varepsilon_2)$ and the state variables $(S,I)$ for $j = 1, 2$, satisfying
\begin{itemize}
    \item $I = 0 \implies h_j(\varepsilon_1, \varepsilon_2, S, 0) = 1$;
    \item $h_j(\varepsilon_1, \varepsilon_2, S, I) = 1 \implies SI = 0$.
\end{itemize}
\end{rem}

While Theorem \ref{prop:endemic-equilibrium} guarantees an endemic equilibrium $x^*$ with positive infected population ($I^* > 0$), it remains to confirm that $x^*$ lies strictly in the interior of the state space $\Omega$. The following proposition demonstrates that positivity of $I^*$ induces a positivity cascade across all remaining state variables.

\begin{prop}\label{prop:interior-equilibrium}
The coordinates of the equilibrium $x^{*}$ in Theorem~\ref{prop:endemic-equilibrium} satisfy $S^{*}, R^{*}, C^{*} > 0$. Hence $x^{*} \in \mathrm{int}(\Om)$.
\end{prop}
\begin{proof}
From $\dot R = 0$ at $x^{*}$:
\[
 (1-\sigma)\beta C^{*} I^{*} + \alpha I^{*} 
    = \left[\mu + (1+\eps_1 g(x^*))\delta\right] R^{*}.
\]
The left-hand side is $\geq \alpha I^{*} > 0$, and the right-hand side is proportional to $R^{*}$ with strictly positive coefficient (since $g(x^*)\geq 0$); hence $R^{*} > 0$.

From $\dot C = 0$ at $x^{*}$:
\[
    (1+\eps_1 g(x^*))\delta R^{*} 
    = \left[\beta I^{*} + \mu + (1+\eps_2 g(x^*))\gamma\right] C^{*}.
\]
The left-hand side is strictly positive since $\delta, R^{*} > 0$, and the right-hand side is proportional to $C^{*}$ with strictly positive coefficient; hence $C^{*} > 0$.

From $\dot S = 0$ at $x^{*}$:
\[
    \mu + (1+\eps_2 g(x^*))\gamma C^{*} 
    = \left[\mu + \beta I^{*}\right] S^{*}.
\]
The left-hand side is at least $ \mu > 0$, and the right-hand side is proportional to $S^{*}$ with strictly positive coefficient, hence $S^{*} > 0$.

Together with $I^{*}  > 0$,
all four coordinates of $x^{*}$ are strictly positive, so $x^{*} \in \mathrm{int}(\Om)$.
\end{proof}

To conclude this section, we establish explicit lower and upper bounds for the endemic equilibrium components $S^*$ and $I^*$. In particular, $I^*$ is bounded relative to $\infEE[\infty] \coloneqq 1 - \mathcal{R}_0^{-1}$, the classical endemic level without feedback. This benchmark will prove essential in Section~\ref{sec:strong_feedback}, where we analyze the system's behavior under strong feedback and demonstrate, among other findings, that $I_\infty^*$ acts as a supremum for $I^*$.

\begin{lemma}
\label{lemma:bounds EE}
    Let $\BRN>1$ and let $g$ satisfy Definition \ref{dfn:funcg}. 
    Then, for all $\eps_1,\eps_2\geq 0$, any endemic equilibrium of SIRCm satisfies
    \begin{align*}
        \frac{\mu}{\mu+\beta}\leq S^*\leq \BRN^{-1},\quad 
        \frac{\mu}{\mu+\alpha}\infEE[\infty]\leq \infEE[]\leq \infEE[\infty],
    \end{align*}
    where $\infEE[\infty]\coloneqq1-\BRN^{-1}$.
\end{lemma}
\begin{proof}
   Since $I^*>0$ for every endemic equilibrium, from $\dot I =0$ in  
   \eqref{eq:SIRCm I} we get
   \begin{align}
    \label{eq: relation SC endemic}
    S^* = \BRN^{-1}-\sigma C^*.
\end{align}
Substituting into $I^*=1-S^*-R^*-C^*$ yields
\begin{align}
\label{eq: relation ISC endemic}
    I^* 
    = \left(1-\BRN^{-1}\right)-R^*-(1-\sigma)C^*,
\end{align}
which, together with \eqref{eq: relation SC endemic}, gives us the upper bounds.
From $\dot S=0$, we have
\begin{align*}
    S^*(\mu+\beta \infEE[])=\mu+(1+\eps_2 g)\gamma C^* \geq \mu,
\end{align*}
hence
$S^*\geq 
\mu/(\mu+\beta \infEE[])
\geq
\mu/(\mu+\beta)$.
Similarly, from \eqref{eq: relation ISC endemic}
and since $S^*\leq \BRN^{-1}$,
we obtain
\begin{align*}
    \infEE[]\geq \frac{\mu (1-S^*)}{\beta S^*}\geq 
    \frac{\mu(1-\BRN^{-1})}{\beta\BRN^{-1}} = 
    \frac{\mu}{\mu+\alpha}\infEE[\infty].
\end{align*}
\end{proof}

\section{Stability analysis and bifurcation study}
\label{sec: stability and bifurcation}

The mutation feedback in SIRCm \eqref{eq:SIRCm} is governed by the parameters $\eps_{1}$ and $\eps_{2}$, present as amplification factors of the immune-escape and waning terms
\[
    (1+\eps_{1}g(x))\quad \text{and} \quad (1+\eps_{2}g(x))
\]
Since these always appear in a product with $g(x)$, it is convenient to introduce the shorthand
\begin{equation}
    \epstilde_{1} \coloneqq    {g(\EE[0])}\,\eps_{1}, \qquad
     \epstilde_{2}\coloneqq g(\EE[0])\,\eps_{2},
    \label{eq:epstilde}
\end{equation}
where $\EE[0]$ is the endemic equilibrium
of the SIRC model.
With this notation, the long-term dynamics of the model will be similar to that of SIRC if
$|\eps_{1,2}g(x)|\ll1$.  If this is not the case,
then substantially different behavior may arise.

This can be observed in Figure \ref{fig:eps1_eps2_prevalence}.
The left panel shows the endemic prevalence $I^*$ of the transmission-driven model as a function of the normalized feedback parameters $\tilde\eps_1$ and $\tilde\eps_2$. The plot reveals two regimes separated by a sharp transition: a low prevalence regime (bottom left), where $I^*$ remains close to its SIRC value, and a high prevalence regime (top right), where $I^*$ is near its theoretical ceiling. The right panel shows the values along the diagonal $\tilde\eps_1 = \tilde\eps_2 = \tilde\eps$.  There we also plot for comparison some asymptotic approximations developed in Section \ref{sec:low_prev_series}. This figure is based on the transmission-driven model; the prevalence-driven model exhibits the same qualitative two-regime structure.
\begin{figure}
    \centering
    \includegraphics[width=1\textwidth]{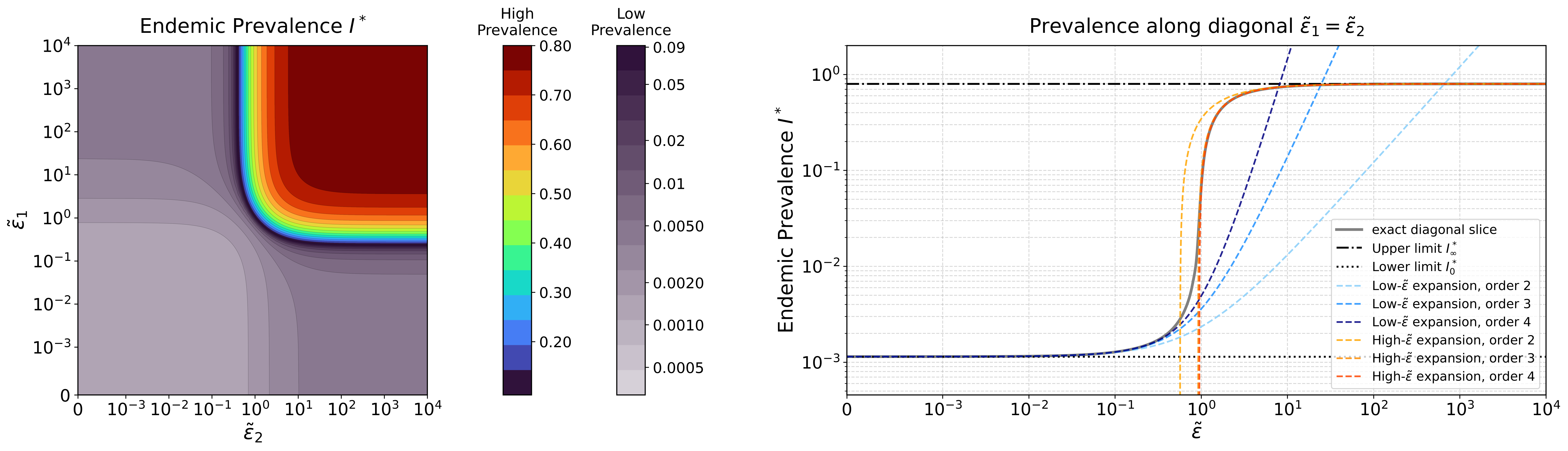}
    \caption{Endemic prevalence $I^*$ of the transmission-driven SIRCm model. \emph{Left}: two-dimensional plot of $I^*$ as a function of $\eps_1$ and $\eps_2$. \emph{Right}: diagonal slice $\epstilde_1 = \epstilde_2 = \epstilde$  together with the first three orders of the Taylor series approximations for low and high feedback regimes presented in Sections \ref{sec:low_prev_series} and \ref{sec:high_prev_series}. The dotted line marks the lower limit $\infEE[0]$, the endemic prevalence of the SIRC model recovered as $\epstilde \to 0 $; the dash-dotted line marks the upper limit $I_\infty^* = 1 - \mathcal{R}_0^{-1}$, the theoretical ceiling on prevalence approached as $\epstilde \to \infty$, corresponding to the SIS limit of Section \ref{sec:strong_feedback}. The color regimes transition at the inflection point of the endemic prevalence along the diagonal (where the derivative of $I^*$ along $\eps_1 = \eps_2$ is maximized)}
    \label{fig:eps1_eps2_prevalence}
\end{figure}

As $(\epstilde_1,\epstilde_2,\beta_0)$ vary, the locus of points at which the complex-conjugate eigenvalue pair of $J(\EE)$ crosses the imaginary axis traces out a two-dimensional Hopf bifurcation surface in this three-parameter space \citep{kuznetsov2023elements}. In order to analyze the stability of the endemic equilibrium, Figure \ref{fig:eps1_eps2_hopf_surface_I} shows this surface for the prevalence-driven model, with all other parameters fixed at the base values. An unstable region is present for sufficiently small $\epstilde_1$. The diagonal $\eps_1 = \eps_2$, however, does not enter this unstable region for any  $\beta_0 \in [200,2000]$. 

\begin{rem}[Notation]
We write $\EE$ for the endemic equilibrium of SIRCm \eqref{eq:SIRCm} at a given value of the  feedback parameter(s). The disease-free equilibrium $\DFE$ is as in Proposition~\ref{prop:DFE}. We further write $\EE = (S^*,I^*,R^*,C^*)^\top$,  where the coordinates inherit the same implicit dependence on $\eps$. Since $I^*$ is the primary object of interest below, we write it explicitly as $I^*_\varepsilon$ when its $\varepsilon$-dependence is the focus.
\end{rem}

\begin{figure}
    \centering
    \includegraphics[width=1\textwidth]{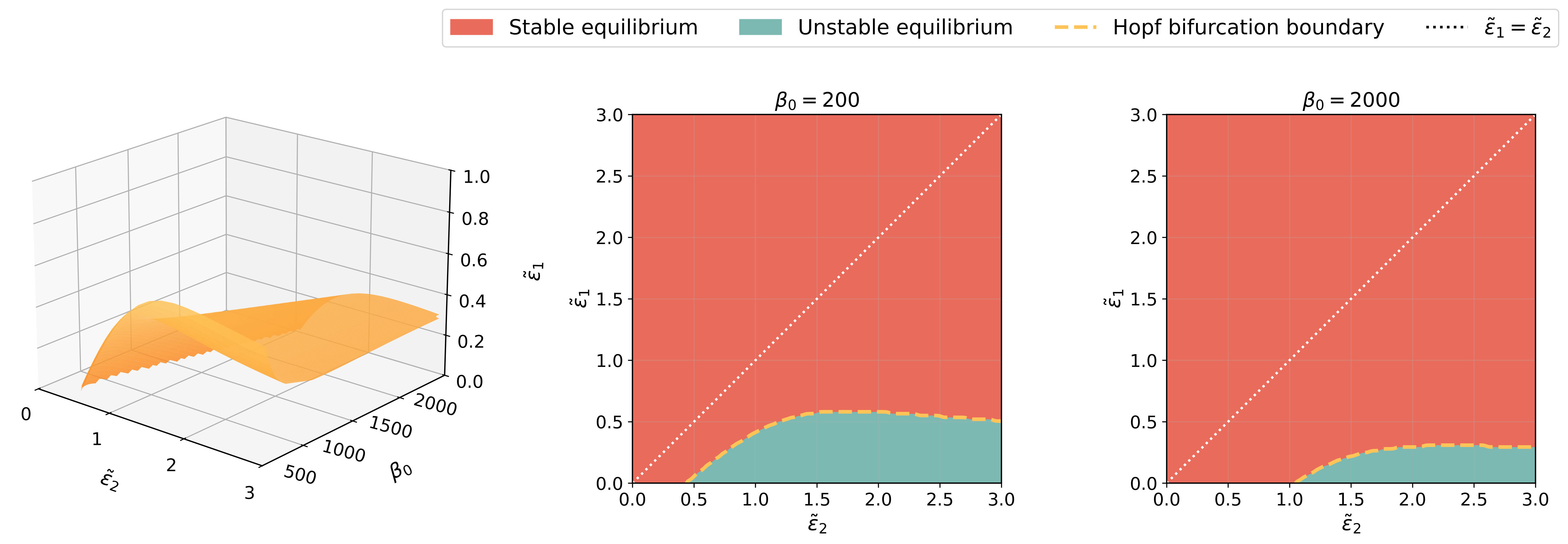}
    \caption{(see Section \ref{sec: stability and bifurcation}) Hopf bifurcation surface across $(\tilde\varepsilon_1,\tilde\varepsilon_2,\beta_0)$ for the prevalence-driven model (left). The center and right panels show fixed-$\beta_0$ cross-sections of this surface at $\beta_0 = 200$ and $\beta_0 = 2000$, respectively, in which the dashed curve marks the Hopf bifurcation boundary; all other parameters are fixed}
    \label{fig:eps1_eps2_hopf_surface_I}
\end{figure}

For fixed $\beta_0$, the locus of Hopf points in the $(\tilde\eps_1,\tilde\eps_2)$-plane forms a curve. To trace it, we first locate a seed point on the $(\tilde\eps_1,\tilde\eps_2)$-grid at which the endemic equilibrium is unstable; a one-parameter pseudo-arclength continuation (PALC) from this seed, together with eigenvalue-crossing detection, then locates an initial Hopf point, from which two-parameter continuation in $(\tilde\eps_1,\tilde\eps_2)$ (all performed in Julia using \textsf{BifurcationKit.jl}, \cite{veltz2020bifurcationkit}) traces the full curve. Repeating this procedure across values of $\beta_0\in[200,2000]$ and interpolating between the resulting curves yields the two-dimensional Hopf bifurcation surface in $(\tilde\eps_1,\tilde\eps_2,\beta_0)$-space shown in Figures~\ref{fig:eps1_eps2_hopf_surface_I}, \ref{fig:eps1_eps2_hopf_surface}, and \ref{fig:hopf_surface}; fixed-$\beta_0$ cross-sections of this surface recover Hopf curves in the $(\tilde\eps_1,\tilde\eps_2)$-plane.

Figure~\ref{fig:eps1_eps2_hopf_surface} shows the analogous Hopf bifurcation surface for the transmission-driven model, with slices at $\beta_0=600$ and $\beta_0=900$. In contrast to the prevalence-driven case, the diagonal $\eps_1=\eps_2$ does cross into the unstable region once $\beta_0$ is large enough. 
\begin{figure}
    \centering
    \includegraphics[width=1\textwidth]{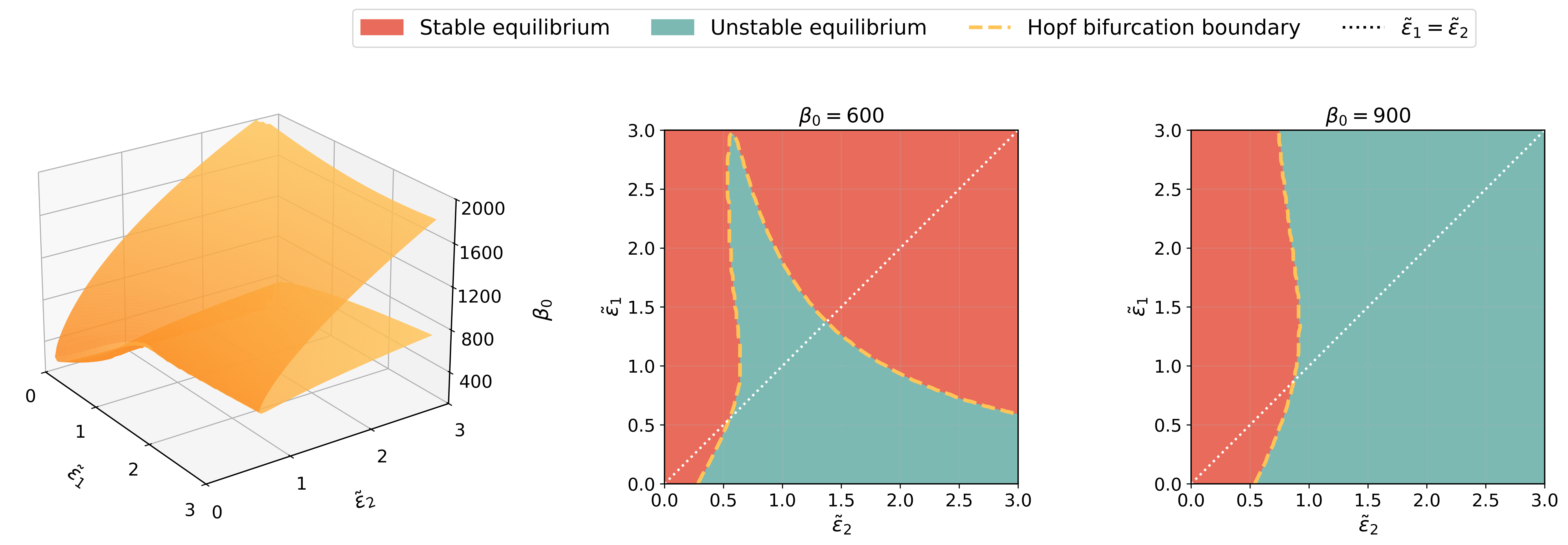}
    \caption{(see Section \ref{sec: stability and bifurcation}) Hopf bifurcation surface across $(\tilde\varepsilon_1,\tilde\varepsilon_2,\beta_0)$ for the transmission-driven model (left). The center and right panels show fixed-$\beta_0$ cross-sections of this surface at $\beta_0 = 600$ and $\beta_0 = 900$, respectively, in which the dashed curve marks the Hopf bifurcation boundary; all other parameters are fixed}
    \label{fig:eps1_eps2_hopf_surface}
\end{figure}

Having seen that the diagonal $\epstilde_1 = \epstilde_2$ remains stable for the prevalence-driven model (Figure \ref{fig:eps1_eps2_hopf_surface_I}) but crosses into the unstable region for the transmission-driven model once $\beta_0$ is large enough (Figure \ref{fig:eps1_eps2_hopf_surface}), we adopt this single-parameter restriction $\epstilde_1 = \epstilde_2 = \epstilde$ for the remainder of the section, as it gives a natural one-parameter family suitable for direct comparison between the two models.

Because the transmission-driven model exhibits a Hopf bifurcation along the diagonal, we track the four eigenvalues of $J(\EE[\eps])$ as functions of $\tilde\eps$ at fixed $\beta_0$. Figure~\ref{fig:eigenvalues} shows this sweep at $\beta_0=600$, where two eigenvalues remain negative throughout, while the remaining complex-conjugate pair crosses the imaginary axis twice, marking the two Hopf bifurcation points that bound the unstable region and partition $[0,2]$ into the three regimes: the endemic equilibrium is locally asymptotically stable for smaller values of $\epstilde$, unstable for the middle region and stable again for larger values. The subsections that follow will examine each of these three regimes (Sections~\ref{sec:weak_feedback}, \ref{sec:moderate feedback}, and \ref{sec:strong_feedback}).
\begin{figure}
    \centering
    \includegraphics[width=1\textwidth]{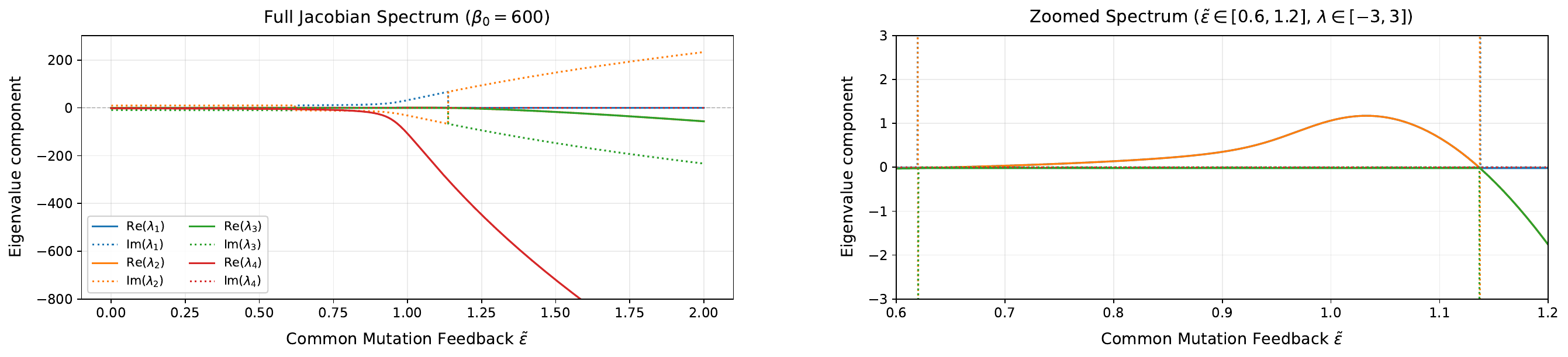}
    \caption{(see Section \ref{sec: stability and bifurcation}) Eigenvalues of $J(x^{*})$ as functions of $\epstilde$, at fixed $\beta_0=600$. \emph{Left:} real and imaginary parts of all four eigenvalues over $\epstilde\in[0,2]$. \emph{Right:} zoom on the complex-conjugate pair whose real part changes sign, showing the two transversal crossings that mark the Hopf bifurcation points}
    \label{fig:eigenvalues}
\end{figure}

Figures~\ref{fig:3d_trajectories} and~\ref{fig:phase_portraits} show the consequences of this loss and recovery of stability in the phase space of the system for the transmission-driven model at one representative value of $\tilde\eps$ from each of the three regimes identified in Figure~\ref{fig:eigenvalues}. At $\tilde\eps=0.5$, trajectories spiral directly into the endemic equilibrium $\EE$. At $\tilde\eps=1.5$, inside the bifurcation region, the equilibrium is unstable and trajectories are instead attracted to a closed orbit encircling it, producing sustained oscillation in prevalence. At $\tilde\eps=1.7$, stability is recovered and trajectories once again spiral into the equilibrium.

This sustained oscillation is notable because it arises without any seasonal forcing in the model; the periodicity is generated purely by the mutation feedback mechanism, in contrast to
the externally imposed seasonality typically invoked to explain recurrent epidemic waves.
\begin{figure}
    \centering
    \includegraphics[width=0.95\textwidth]{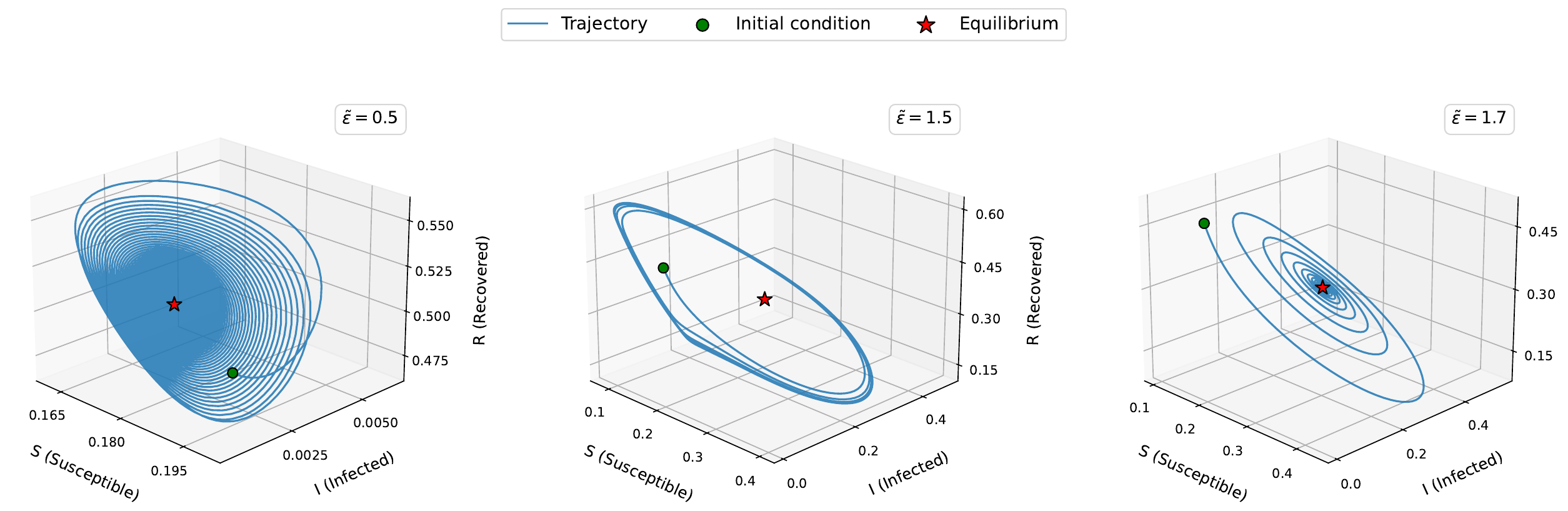}
    \caption{(see Section \ref{sec: stability and bifurcation}) Trajectories in the $(S,I,R)$ space for three representative values of $\tilde\varepsilon$, spanning the weak, moderate and strong feedback regimes for the transmission-driven model}
    \label{fig:3d_trajectories}
\end{figure}

\begin{figure}
    \centering
    \includegraphics[width=0.8\textwidth]{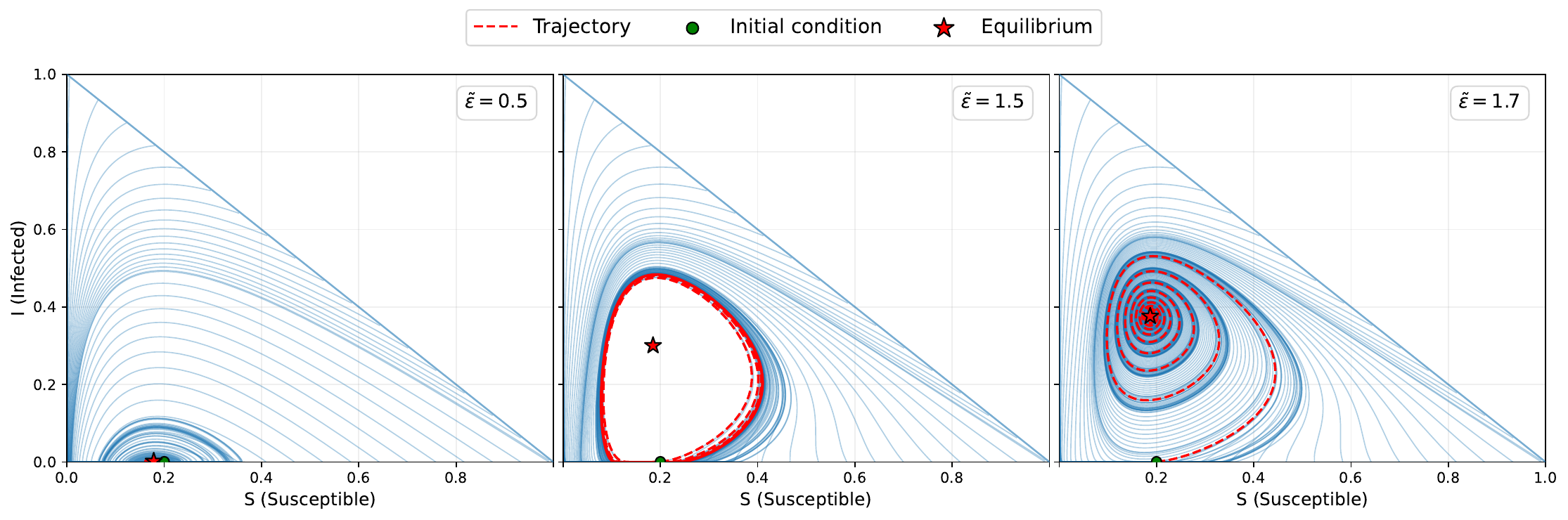}
    \caption{(see Section \ref{sec: stability and bifurcation}) Phase portraits for the same three values of $\epstilde$ as Figure~\ref{fig:3d_trajectories}, projected onto the $(S,I)$ plane for the transmission-driven model}
    \label{fig:phase_portraits}
\end{figure}

For the prevalence-driven model, no comparable transition occurs along the diagonal.
Trajectories and phase portraits converge directly to the endemic equilibrium at every value of $\tilde\varepsilon$ tested, as shown in Figures~\ref{fig:3d_trajectories_I} and
\ref{fig:phase_portraits_I} for three representative values of
$\tilde\varepsilon$.

\begin{figure}
    \centering
    \includegraphics[width=0.95\textwidth]{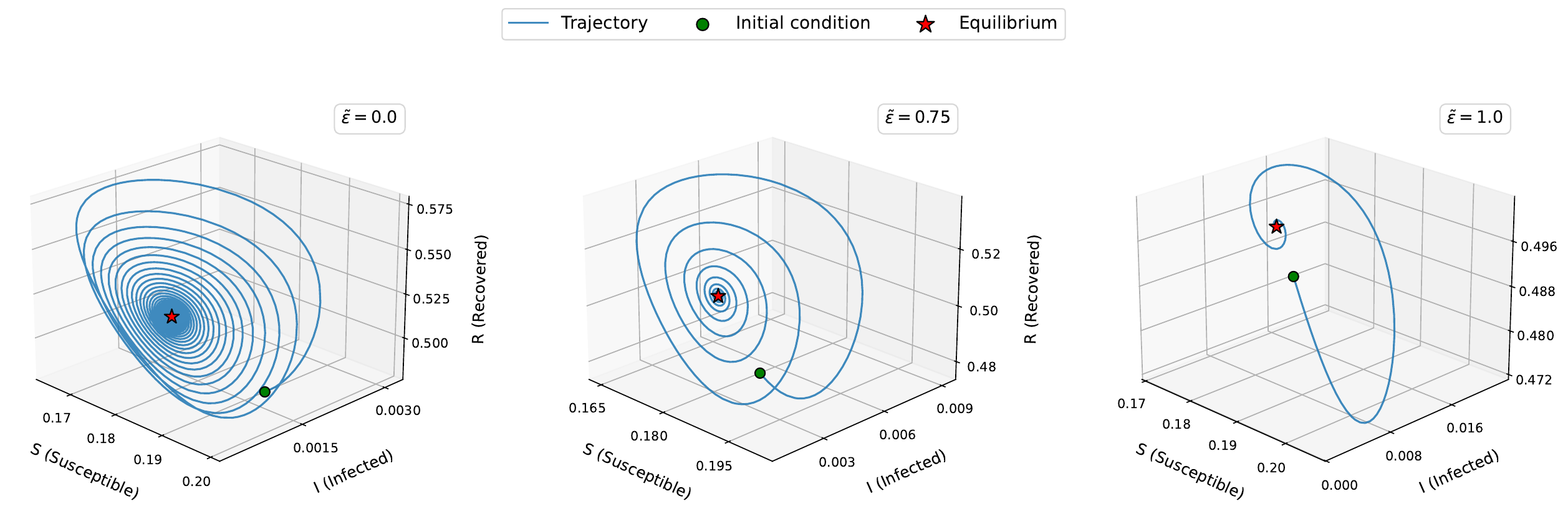}
    \caption{(see Section \ref{sec: stability and bifurcation}) Trajectories in $(S,I,R)$ space at $\tilde\varepsilon = 0, 0.75, 1$ for the prevalence-driven model. All trajectories converge to the endemic equilibrium; trajectories for $\tilde\varepsilon>1$ behave similarly to those for $\tilde\varepsilon=1$} 
    \label{fig:3d_trajectories_I}
\end{figure}

\begin{figure}
    \centering
    \includegraphics[width=0.8\textwidth]{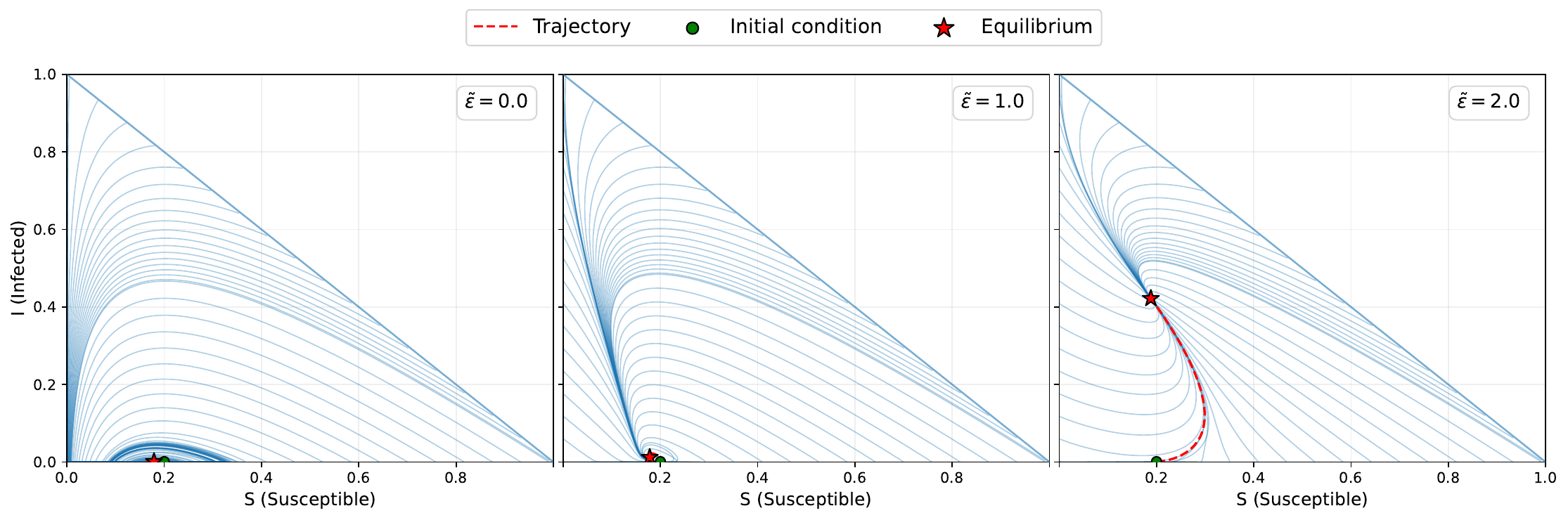}
    \caption{(see Section \ref{sec: stability and bifurcation}) Phase portraits in the $(S,I)$ plane at $\tilde\varepsilon = 0, 1, 2$ for the prevalence-driven model}
    \label{fig:phase_portraits_I}
\end{figure} 

The results presented in the remainder of this section can be summarized as follows.
\begin{itemize}
    \item For weak enough or strong enough mutation feedback, i.e., $0<\eps\ll 1$ or $\eps \gg 1$,
    the endemic equilibrium is unique and
    asymptotically stable (for both transmission-
    and prevalence-driven models).
    \item For the transmission-driven model, when mutation and waning effects balance each other, i.e., $\eps=\mathcal{O}(1)$, 
    numerical evidence strongly suggests that the endemic equilibrium is unique and $I^*$
    is a monotonically increasing function of $\eps$. Moreover, a detailed bifurcation study shows that the endemic equilibrium can be
    unstable in this regime.
    \item For the prevalence-driven model, by contrast, we find no such loss of stability along the diagonal $\eps_1 = \eps_2$ for any parameter values tested; the endemic equilibrium remains stable throughout.
\end{itemize}

\subsection{Weak mutation feedback}
\label{sec:weak_feedback}

\begin{prop}
\label{prop: small epsilon}
    Let $\BRN>1$ and let $g$ satisfy Definition \ref{dfn:funcg}. Then there exists an ${{\eps_0}}>0$ such that the SIRCm system \eqref{eq:SIRCm} has 
    a unique, asymptotically stable, endemic equilibrium for $\eps\in[0,{\eps_0})$.
\end{prop}
\begin{proof}
    Let us write the right-hand side of SIRCm \eqref{eq:SIRCm} as 
    \begin{align*}
        F(x,\eps) \coloneqq f(x) + \eps g(x) \begin{bmatrix}
            \gamma C\\
            0\\
            -\delta R\\
            \delta R -\gamma C
        \end{bmatrix},
    \end{align*}
    where $f(x)$ is the right-hand side of SIRC.
    Since $\BRN>1$, we know that SIRC has a unique, asymptotically stable endemic equilibrium $\EE[0]$, as shown in \citet[Section 3]{casagrandi2006}. That is, 
    \begin{align}
        \label{eq:stability endemic SIRC}
        F(\EE[0],0)=0, \quad \mathrm{Re}(\lambda_{J_f}(\EE[0]))<0.
    \end{align}
    In particular, we have that ${\partial_x F}(\EE[0],0)=J_f(\EE[0])$ has non-zero determinant.
    Therefore, the implicit function theorem (IFT) guarantees the existence of $\eps_{\text{IFT}}>0$, a neighborhood of $\EE[0]$, $U\subset \Omega$, and a unique continuously differentiable
    mapping $\psi^*:[0,\eps_{\text{IFT}})\to U$ such that $\psi^*(0)=\EE[0]$ and
    \begin{align*}
        F(\psi^*(\eps),\eps) = 0,
    \end{align*}
    for $\eps \in [0,\eps_{\text{IFT}})$.
    
    Let us now show the existence of some $\underline{\eps}>0$, such that
    $\psi^*(\eps)$ is the only endemic equilibrium in $\Omega$ (rather than just in $U$) for $0<\eps<\underline{\eps}$.
    For the sake of contradiction, we assume the opposite, i.e., 
    there is a sequence $\eps_n\subset \real^+$ and 
    $x^*_{\eps_n}\subset \Omega$ of
    endemic equilibria such that $\eps_n\to 0$ as $n\to+\infty$ and $x^*_{\eps_n}\neq \psi^*(\eps_n)$ 
    for all $n\in \naturals$. Since $\Omega$ is compact, we can extract a convergent subsequence of 
    $x^*_{\eps_n}$, which we still index by $n$ for simplicity, with limit $x^*\in\Omega$.
    By the continuity of $F$, we  have
    \begin{align*}
        F(x^*,0)=\lim_{n\to\infty}F(x_{\eps_n},\eps_n)=0.
    \end{align*}
    Therefore $x^*\in\{\EE[0], \DFE \}$, since SIRC only possesses two equilibria.
    From \eqref{eq: relation SC endemic}, we know that the continuous 
    function $\theta(x)\coloneqq S+\sigma C$ satisfies $\theta(x^*_{\eps_n})=\BRN^{-1}$ for every 
    $n\in \naturals$.
    If $x^*=\DFE$, we get $1=\theta(\DFE)=\lim_{n\to \infty}
    (x^*_{\eps_n})=\BRN^{-1}$, which is a contradiction.
    Alternatively, if $x^*=x_0^*$, then there is a neighborhood $V\subset\Omega$
    of $x^*_0$ such that $x^*_{\eps_n}\in V\cap U$ and $\eps_n<\eps_{\text{IFT}}$ for $n$ large enough.
    However, $\psi^*(\eps)$ is unique in $U\cap V$ for $\eps<\eps_{\text{IFT}}$, and we get a contradiction.

    Finally, let us take $0<\eps<\underline{\eps}$ and denote by
    $\lambda_\eps$ the eigenvalues of $\partial_x F(\psi^*(\eps),\eps)$,
    given by the roots of the polynomial $p_\eps(\lambda)\coloneqq\mathrm{det}\left(\lambda \,\mathrm{Id}-\partial_x F(\psi^*(\eps),\eps)\right)$.
    Since $\psi^*$ is smooth, the coefficients of $p_\eps$, and therefore its roots,
    depend continuously in $\eps$.
    Hence, from \eqref{eq:stability endemic SIRC},
    there is some
    ${\eps_0}>0$ such that $\mathrm{Re}(\lambda_\eps)<0$ for all $\eps<{\eps_0}$.
\end{proof}
\begin{rem}
Observe that the proof of Proposition \ref{prop: small epsilon}
extends unchanged to the general two-parameter $(\eps_1,\eps_2)$
setting.
\end{rem}

\begin{rem}\label{cubicpolyremark}
For the prevalence-driven model, it is possible to obtain
a quantitative threshold in $\eps$ for the existence of
a unique endemic equilibrium.
Given $\eps>0$, one can verify that once the prevalence $\infEE$ of the endemic equilibrium
is obtained, every other component of the endemic equilibrium can be computed.
Moreover, it can be shown that the prevalence of the endemic equilibrium satisfies
the following algebraic equation 
\begin{align}
\label{eq: incidence poly}
    \poly(\infEE)=a_3 {\infEE}^3+ a_2 {\infEE}^2 +a_1 {\infEE} +a_0 = 0,
\end{align}
where 
\begin{align*}
  a_3 &= \beta\delta\varepsilon\,(\gamma\varepsilon+\beta\sigma), \\
    a_2 &=\alpha  (\beta +\gamma  \eps ) (\beta +\delta  \eps )+\beta ^2 (\mu -\delta  \sigma  (\eps
   -1))+\beta  \eps  (\gamma  ((2-\eps) \delta +\mu )+\delta  \mu  (\sigma +1))+\gamma  \delta  \mu  \eps
   ^2, \\
  a_1 &= \alpha  \beta  (\gamma +\delta +2 \mu )+\alpha  \eps  (\gamma  (2 \delta +\mu )+\delta  \mu )-\beta
   ^2 (\delta  \sigma +\mu )+\\
   &\quad \beta  \gamma  (-2 \delta  \eps +\delta +\mu -\mu  \eps )+\beta  \mu  (\delta  (\sigma
   -\eps +1)+2 \mu )+\gamma  \mu  \eps  (2 \delta +\mu )+\delta  \mu ^2 \eps, \\
  a_0 &= (\gamma +\mu ) (\delta +\mu ) (\alpha -\beta +\mu ).
\end{align*}
Descartes' rule of signs states that the number of sign changes
in the coefficients of $\poly$ will bound the number of positive 
roots from above.
Since $\BRN=\frac{\beta}{\alpha+\mu}>1$, $a_0<0$ and, as long
as $\eps>0$, we also have $a_3>0$. 
Finally, it suffices to require $\eps\leq 1$ for $a_2>0$.
Then, $\poly$ will have a unique positive root $\infEE$
(note that the sign of $a_1$ is irrelevant as there
will  be a single sign change in any case).
From the existence results in Section \ref{sec: persistence theory}, we know
that $x^*\in \Omega$.
In particular, under the scaling \eqref{eq:epstilde}, the additional
condition for uniqueness becomes
\begin{align*}
    \epstilde \leq \infEE[0],
\end{align*}
where $\infEE[0]$ is the prevalence of the
endemic equilibrium of SIRC.

Unfortunately, these computations quickly become intractable
for the transmission-driven model, where
the polynomial $\poly$ 
becomes a quartic with intricate coefficients
(see Appendix \ref{app:quartic}).
\end{rem}

\subsubsection{Expansion of $\infEE$ for small $\eps$}
\label{sec:low_prev_series}
It is desirable to obtain an approximation of the endemic equilibrium
for small values of $\eps$.
For this, we express the constraint \eqref{eq: incidence poly}
as a quadratic polynomial in $\eps$,
\begin{align}
\label{eq:polynomial eps I}
     \poly_0(\infEE)+\eps \poly_1(\infEE) +\eps^2 \poly_2(\infEE) =0,
\end{align}
where $\poly_0$ is a polynomial of degree 2 in $\infEE$, while $\poly_1$ and $\poly_2$ are polynomials of degree 
at most 4 in $\infEE$, depending on the choice of $g$.
From \cite{casagrandi2006}, we know there is a unique 
$I^*_0\in(0,1]$
such that
\begin{align*}
    \poly_0(I^*_0) = a_0+a_1 \infEE[0] +a_2 (\infEE[0])^2 = 0.
\end{align*}
Then 
\begin{align*}
    I^*_0 = \frac{-a_1\pm \sqrt{D}}{2a_2},\qquad D=a_1^2-4a_0 a_2,
\end{align*}
and
\begin{align*}
    \poly_0'(\infEE[0]) = 2a_2 I^*_0+a_1=\pm \sqrt{D}\ne 0, \quad \poly_0''(\infEE[0]) = 2 a_2, \quad \poly_0'''(\infEE[0]) =0;
\end{align*}
the discriminant $D$ above does not vanish for small positive $\varepsilon$ (see Remark \ref{cubicpolyremark}).

For a small enough $\eps$, 
let us consider a (convergent Taylor) series representation of $\infEE$ (the component $I$ of $\psi^*$ from 
the proof of Proposition \ref{prop: small epsilon}):
\begin{align*}
    \infEE = \infEE[0]+ r(\eps)\coloneqq\infEE[0]+ \sum_{k\in \naturals^+} r_k \eps^k,
\end{align*}
which we can plug in \eqref{eq:polynomial eps I}
to collect terms of increasing
powers in $\eps$. 
Given any power series $f(z) = \sum_{k\in\naturals} f_k z^k$, we define the coefficient extractor operator 
$[z^{{n}}]$
\citet[Ch. I.1]{flajolet2009} as
\begin{align*}
    [z^{n}]f(z)\coloneqq f_{n}.
\end{align*}
Since $\poly_0$, $\poly_1$, and $\poly_2$ are polynomials, we can write
\begin{align*}
    \poly_0(\infEE)
    = \poly_0'(\infEE[0]) r +a_2 r^2,\quad 
    \poly_1(\infEE) =  \sum_{k=0}^4 \frac{r^k}{k!}\poly_1^{(k)}(\infEE[0]),\quad
    \poly_2(\infEE) =  \sum_{k=0}^4 \frac{r^k}{k!}\poly_2^{(k)}(\infEE[0]).
\end{align*}
Substituting into \eqref{eq:polynomial eps I} we get
\begin{subequations}
\begin{align}
    \poly_0(\infEE) &= -\eps \poly_1(\infEE) -\eps^2 \poly_2(\infEE)\\
    \Rightarrow
    r &=-\poly_0'(
    \infEE[0])^{-1} \left(a_2 r^2 +\eps \sum_{k=0}^4 \frac{r^k}{k!}\poly_1^{(k)}(\infEE[0]) + \eps^2\sum_{k=0}^4 \frac{r^k}{k!}\poly_2^{(k)}(\infEE[0])\right)\\
    \Rightarrow r_n &= -\poly_0'(
    \infEE[0])^{-1} [\eps^n]\left(a_2 r^2 +\eps \sum_{k=0}^4 \frac{r^k}{k!}\poly_1^{(k)}(\infEE[0]) + \eps^2\sum_{k=0}^4 \frac{r^k}{k!}\poly_2^{(k)}(\infEE[0])\right). \label{eq: recursion expansion coeffs}
\end{align}
\label{eq: general expression series coeffs}
\end{subequations}
Using equation \eqref{eq: recursion expansion coeffs},
it is possible to recursively get analytic---although complicated---expressions
for the coefficients $r_n$. 
The resulting approximations (of order 2, 3, and 4) of $\infEE$ are displayed
in the right panel of Figures \ref{fig:eps1_eps2_prevalence} and \ref{fig:eps1_eps2_prevalence_I} for the transmission-driven (\ref{t-dv_formula}) and prevalence-driven (\ref{p-dv_formula})
models, respectively.

\subsection{Moderate mutation feedback---bifurcation region}
\label{sec:moderate feedback}

Figure \ref{fig:prevalence_comparison} traces the endemic prevalence as a function of $\beta_0$ for several values of $\epstilde$, for the transmission-driven model. At small $\epstilde$ (left panel), each curve rises from $\BRN = 1$ and remains stable throughout. At larger $\epstilde$ (right panel), each curve develops an unstable segment, bounded by two Hopf bifurcation points. 
\begin{figure}
    \centering
    \includegraphics[width=0.95\textwidth]{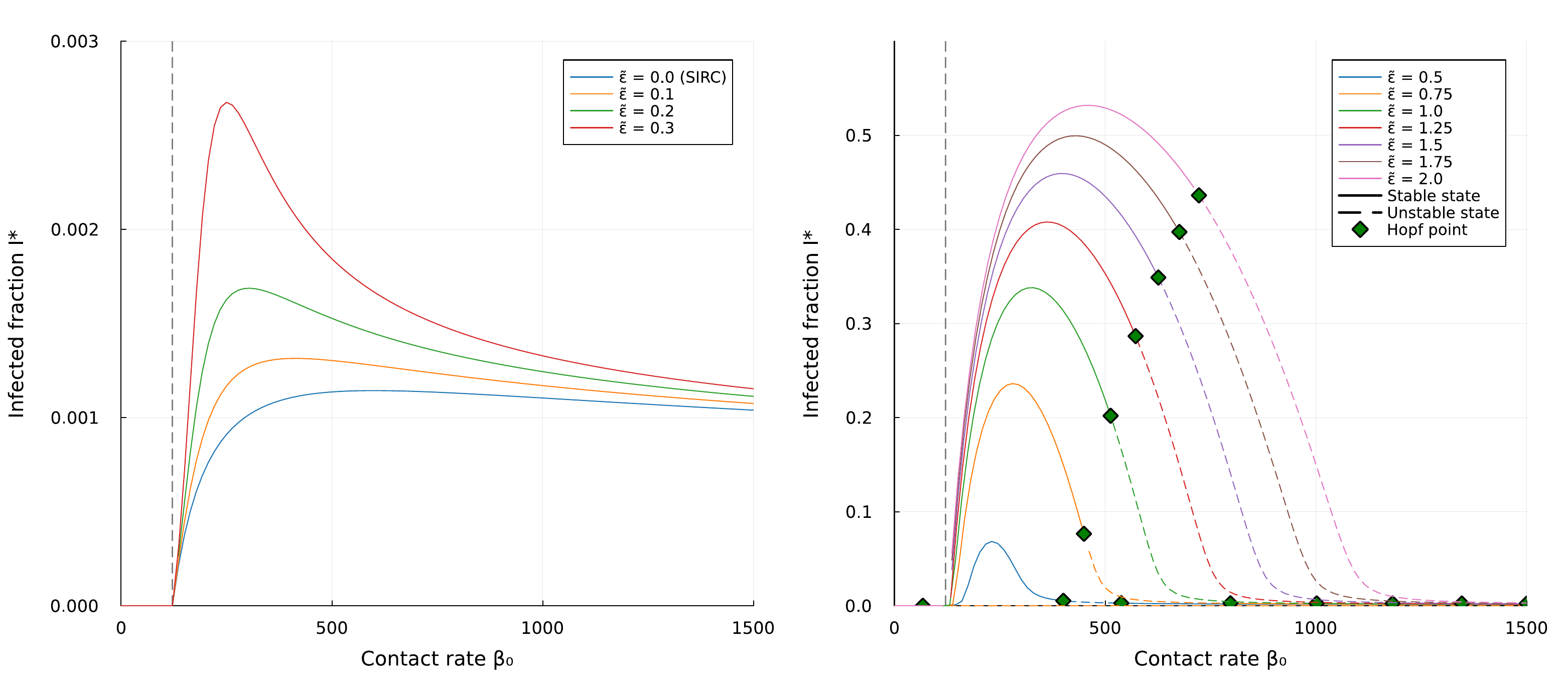}
    \caption{(see Section \ref{sec:moderate feedback}) Endemic prevalence  $I^*$ of the transmission-driven model against contact rate $\beta_0$ under varying mutation feedback strengths ($\epstilde$). Solid and dashed curves denote stable and unstable equilibria, respectively. Note the scale difference in the $y$-axis}
    \label{fig:prevalence_comparison}
\end{figure}

The analogous plot for the prevalence-driven model (Figure \ref{fig:prevalence_comparison_I}) shows no such unstable segments, and each curve retains the same qualitative shape as the SIRC model, simply attaining a higher peak prevalence as $\epstilde$ increases. \begin{figure}
    \centering
    \includegraphics[width=0.45\textwidth]{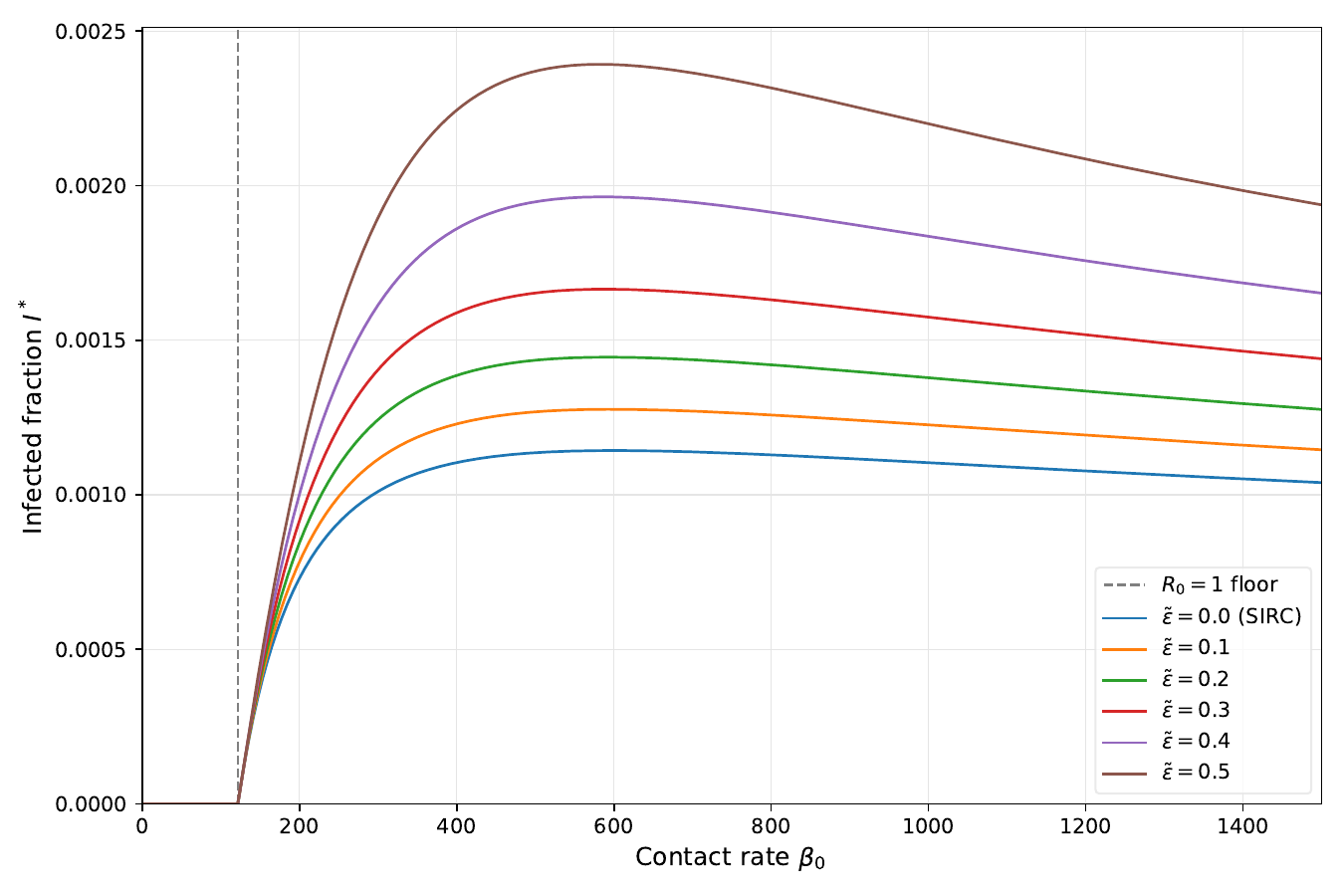}
    \caption{(see Section \ref{sec:moderate feedback}) Endemic prevalence $I^*$ against contact rate $\beta_0$ for different values of $\epstilde$}
    \label{fig:prevalence_comparison_I}
\end{figure}

The eigenvalue sweep of Figure~\ref{fig:eigenvalues} was computed at a single, fixed pair of parameter values, $\beta_{0}=600$ and $\sigma=0.07874$. Two natural questions follow: how do the two Hopf bifurcation points move as $\beta_{0}$ and $\sigma$ vary, and does the qualitative three-regime picture persist across the ranges of these parameters relevant to influenza-like transmission? Figure~\ref{fig:hopf_surface} answers this by tracing the Hopf bifurcation curves via numerical continuation.
\begin{figure}
    \centering    \includegraphics[width=1\textwidth]{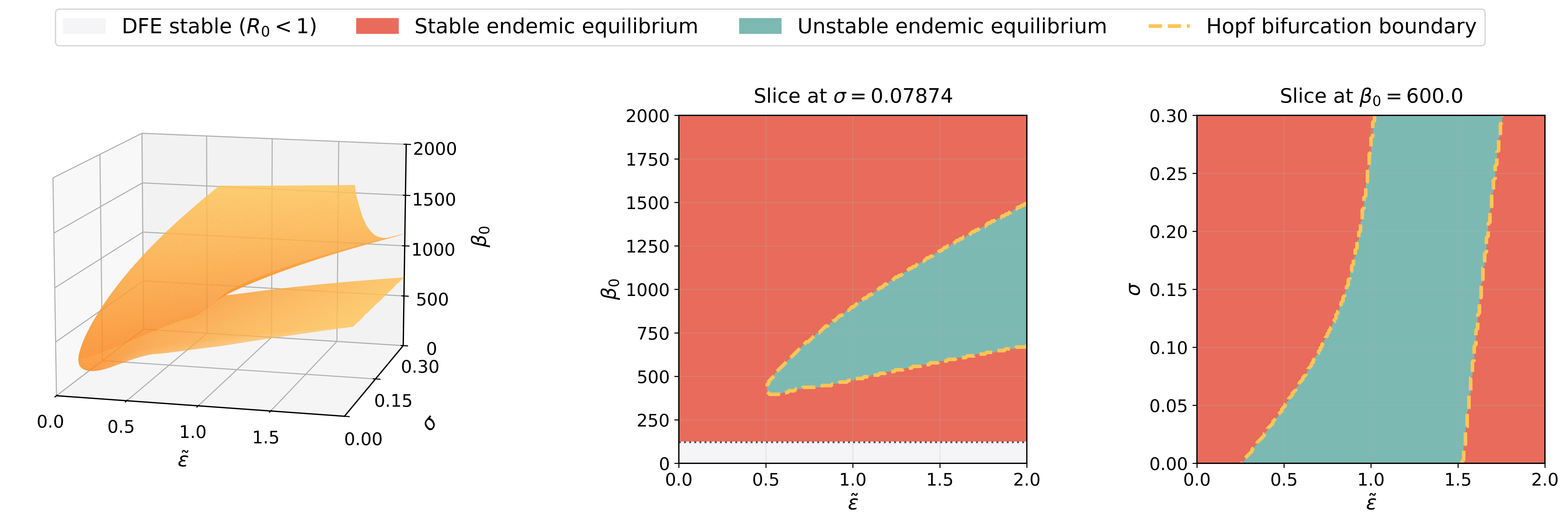}
    \caption{(see Section \ref{sec:moderate feedback}) Stability regions of the endemic equilibrium of the transmission-driven model across the three-parameter space $(\epstilde,\beta_0,\sigma)$, computed via numerical continuation of the Hopf bifurcation curve}
    \label{fig:hopf_surface}
\end{figure}

Figure \ref{fig:eps1_eps2_hopf_surface} shows that the two Hopf bifurcation points along the diagonal at $\beta_0 = 600$ occur at $\epstilde \approx 0.63$ and $\epstilde \approx 1.58$. Figure \ref{fig:timeseries_SI} illustrates the resulting change in the dynamics of $I(t)$ directly. We select $\epstilde = 0.25$, well below the first bifurcation point, where the endemic equilibrium is stable and $I(t)$ settles into a fixed prevalence above SIRC after a few damped oscillations. We then take $\epstilde = 1.57$, just below the second bifurcation point, so the equilibrium is still unstable and the trajectory shows sustained, large-amplitude oscillations. Finally, we take $\epstilde = 1.59$ just past the second bifurcation point, where stability has only just been recovered and $I(t)$ decays again toward the endemic equilibrium.
\begin{figure}
    \centering \includegraphics[width=1\textwidth]{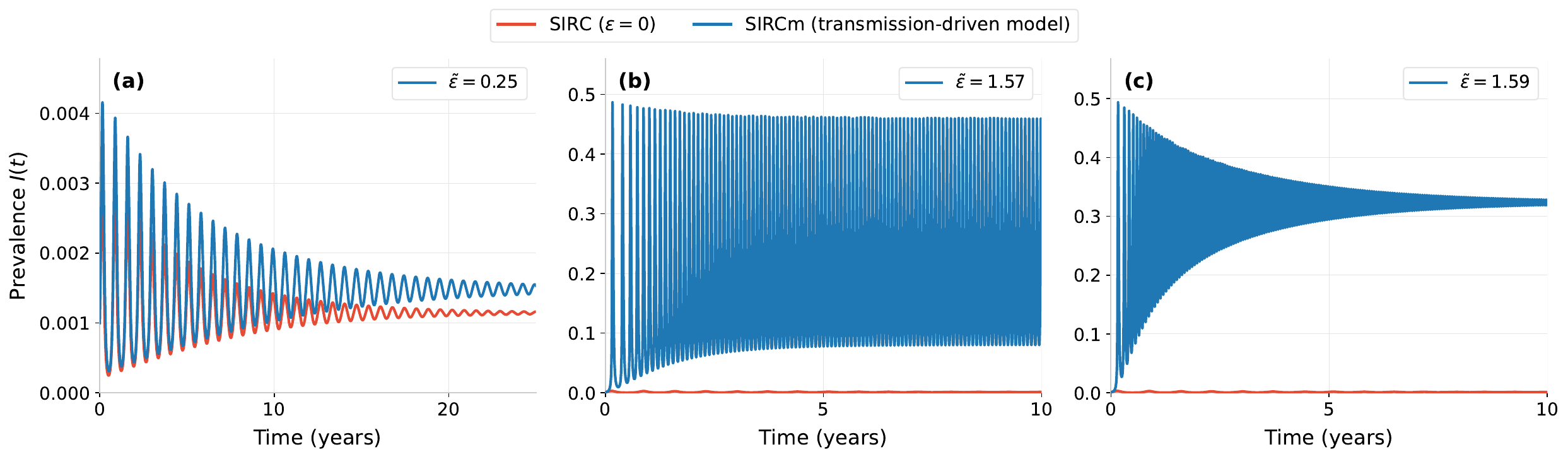}
    \caption{(see Section \ref{sec:moderate feedback}) Time series of the infected fraction $I(t)$ for the transmission-driven model at $\beta_0 = 600$, compared against SIRC}
    \label{fig:timeseries_SI}
\end{figure}

For the prevalence-driven model, no comparable bifurcation structure exists. Figure \ref{fig:eps1_eps2_hopf_surface_I} shows that the diagonal $\epstilde_1 = \epstilde_2$ never enters the unstable region for any $\beta_0 \in [200,2000]$. Figure \ref{fig:timeseries_I} illustrates instead how the rate of convergence to equilibrium changes with feedback strength. As $\epstilde$ increases, the endemic prevalence rises above its SIRC baseline, while convergence to the equilibrium becomes faster. 
\begin{figure}
    \centering \includegraphics[width=1\textwidth]{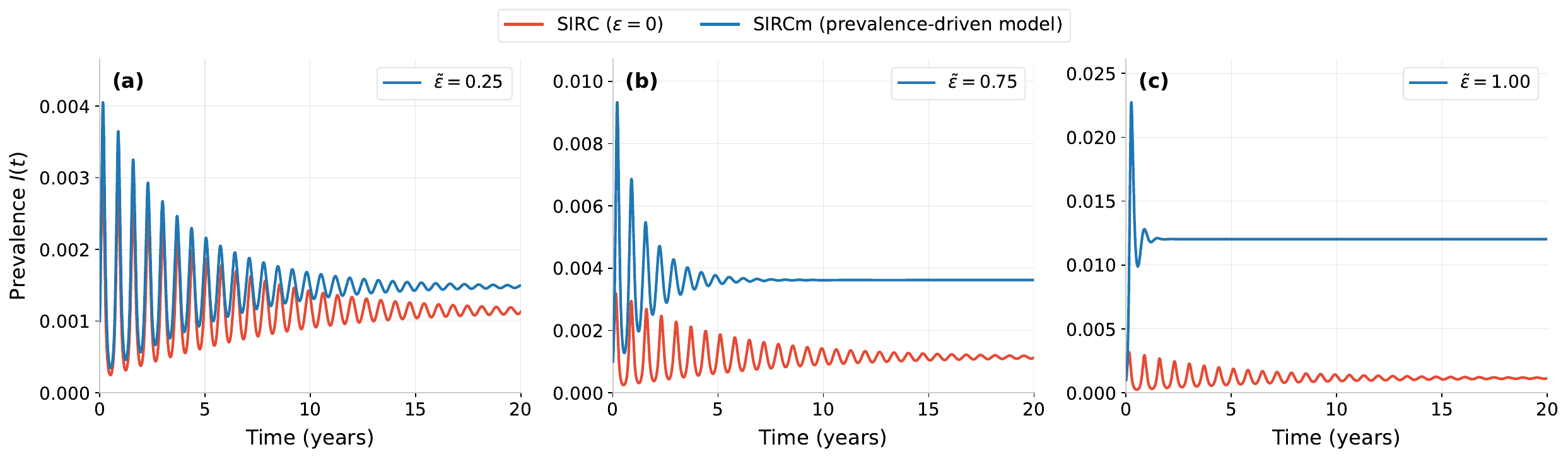}
    \caption{(see Section \ref{sec:moderate feedback}) Time series of the infected fraction $I(t)$ for the prevalence-driven model at $\beta_0 = 600$, compared against SIRC}
    \label{fig:timeseries_I}
\end{figure}

\subsection{Strong mutation feedback---beyond the bifurcation region}
\label{sec:strong_feedback}
Let us study the behavior of the SIRCm model \eqref{eq:SIRCm} for $\eps\gg 1$. 
Recall that, for every $\eps$, the endemic equilibrium satisfies
\begin{align}
\label{eq: R and C endemic}
    R^* = \frac{\alpha I^* + (1-\sigma )\beta C^* I^*}{\mu+(1+\eps g)\delta},\quad
    C^* = \frac{\delta (1+\eps g)R^*}{\beta I^*+\mu+(1+\eps g)\gamma}.
\end{align}
Then, the following lemma provides us with some useful uniform (in $\eps$)
estimates on the endemic equilibrium. 
In particular, we establish that $R^*$ and $C^*$ must vanish
as $\eps\to\infty$.
\begin{lemma}
\label{lemma:uniform scaling}
Let $\BRN>1$ and let $g$ satisfy Definition \ref{dfn:funcg} and define 
\begin{align*}
    \mathcal{W}\coloneqq\left[\tfrac{\mu}{\mu+\beta},\,\BRN^{-1}\right]
    \times\left[\tfrac{\mu}{\mu+\alpha}\infEE[\infty],\,\infEE[\infty]\right],
    \qquad
    g_{\min}\coloneqq\min_{(S,I)\in\mathcal{W}} g(x). 
\end{align*}
Then $g_{\min}>0$ and, for every $\eps>0$, every endemic equilibrium
$\EE=(S^*,I^*,R^*,C^*)$ of SIRCm \eqref{eq:SIRCm} satisfies
\begin{gather}
g(\EE)\ \geq\ g_{\min}\ >\ 0; \label{eq: gmin bounded} \\
0<C^*\leq \tfrac{\delta}{\gamma}\ R^*\leq \epsinv \frac{\alpha+\beta}{{\gamma}\,g_{\min}}. \label{eq: rfast cfast vanish}
\end{gather}
As a consequence $R^*,C^*=\bigo(\eps^{-1})$.
\end{lemma}
 
\begin{proof}
By Lemma \ref{lemma:bounds EE}, $(S^*,I^*)\in\mathcal{W}$ for every $\eps>0$,
and \eqref{eq: gmin bounded}  follows from the fact
that $\mathcal{W}$ is compact (hence $g_{\min}$ is attained) and $S,I>0$ in  $\mathcal{W}$.

For \eqref{eq: rfast cfast vanish}, we rewrite the first identity of \eqref{eq: R and C endemic} as
\begin{align*}
    R^*\left[\mu+\bigl(1+\eps g(\EE)\bigr)\delta\right]
    = \alpha I^* + (1-\sigma)\beta C^* I^*.
\end{align*}
The right-hand side is bounded from above by $\alpha+\beta$,
while the left-hand side is bounded from below by $\eps\, g_{\min}\delta  R^*$
using \eqref{eq: gmin bounded}. 
Hence
\begin{align*}
    R^*\leq\frac{\alpha+\beta}{\delta\,g_{\min}}\eps^{-1}.
\end{align*}
Similarly, the second identity of \eqref{eq: R and C endemic} gives
\begin{align*}
    C^*=\frac{\delta\bigl(1+\eps g(\EE)\bigr)R^*}
             {\beta I^*+\mu+\bigl(1+\eps g(\EE)\bigr)\gamma}
      \leq\frac{\delta\bigl(1+\eps g(\EE)\bigr)R^*}
             {\bigl(1+\eps g(\EE)\bigr)\gamma}
      =\frac{\delta}{\gamma}R^* .
\end{align*}
Strict positivity of $R^*$ and $C^*$ is given by Proposition
\ref{prop:interior-equilibrium}.
\end{proof}

\begin{rem}[Scaled variables] 
\label{rem: scaling notation}
Motivated by Lemma \ref{lemma:uniform scaling}, we introduce the notation
$(\rfast,\cfast)\coloneqq(\eps R, \eps C)$ and
$\xfast \coloneqq (S,I,\rfast,\cfast)$, where $x=(S,I,R,C)$, noting 
that the mapping $x\mapsto\xfast$ is linear and one-to-one for $\eps>0$.
Since $g$ only depends on $S$ and $I$, we observe that $g(\xfast)=g(x)$.
Moreover, from Lemma \ref{lemma:uniform scaling}, given $\eps>0$ and an endemic equilibrium $\EE$,
$\EEfast$ lies in
the fixed compact set
\begin{align}
\label{eq: compact box}
    \mathcal{B}\coloneqq\mathcal{W}\times\left[0,\frac{\alpha+\beta}{\delta\,g_{\min}}\right]
    \times\left[0,\frac{\alpha+\beta}{{\gamma}\,g_{\min}}\right]\subset\real^4 .
\end{align}
\end{rem}

From equations
\eqref{eq:SIRCm R} and \eqref{eq:SIRCm C}, we get
\begin{subequations}
\label{eq:SIRCm fast-slow}
\begin{align}
    \dot S &= \mu(1-S) - \beta SI +(\epsinv+  g(\xfast))\gamma\cfast, \\
    \dot I &= \beta SI + \sigma\beta \epsinv \cfast I - (\mu+\alpha)I, \\
    \epsinv \dot \rfast  &= (1-\sigma)\beta \epsinv \cfast I + \alpha I - \epsinv \mu \rfast -
    (\epsinv+  g(\xfast))\delta \rfast, \\
    \epsinv \dot \cfast  &= (\epsinv +  g(\xfast))\delta \rfast - \beta \epsinv \cfast I - \mu \epsinv \cfast  - (\epsinv +  g(\xfast))\gamma \cfast.
\end{align}
\end{subequations}
Moreover, writing \eqref{eq:SIRCm fast-slow} as 
\begin{align*}
    \begin{pmatrix}
        \dot S , \dot I
    \end{pmatrix}^\top = \fslow(\xfast,\epsinv),\quad 
    \epsinv \begin{pmatrix}
        \dot \rfast , \dot \cfast
    \end{pmatrix}^\top = \ffast(\xfast,\epsinv),
\end{align*}
defining the fast time scale $\tau\coloneqq t\eps$, 
and denoting $(\cdot)'=\frac{d(\cdot)}{d\tau}$,
we can express \eqref{eq:SIRCm fast-slow}
as
\begin{align*}
    \begin{pmatrix}
         S' ,  I'
    \end{pmatrix}^\top = \epsinv \fslow(\xfast,\epsinv),\quad 
    \begin{pmatrix}
        \rfast' ,  \cfast'
    \end{pmatrix}^\top = \ffast(\xfast,\epsinv).
\end{align*}

As $\epsinv \to 0$, the fast variables $\bar{R}$ and $\bar{C}$ are algebraically constrained by the slow variables $S$ and $I$, defining the critical manifold 
\begin{align}
    \label{eq:enslaved variables}
    \mathcal{M}_\infty \coloneqq\left\{\xfast \in \real^4:\;
    \bar{R} = \frac{\alpha I}{\delta g(\xfast)}, \quad 
    \bar{C}
    =\frac{\alpha I}{\gamma g(\xfast)}
    \right\}.
\end{align}
Substituting the constraints defining $\critman$ into $\fslow(\xfast,0)$
yields the
\emph{reduced}, or \emph{slow}, subsystem
\begin{subequations}
\label{eq:SIRCm limit slow subsystem}
\begin{align}
    \dot S &= \mu(1-S) - \beta S I + \alpha I, \\
    \dot I &= I\left[\beta S -(\mu+\alpha)\right].
\end{align}
\end{subequations}
Note that \eqref{eq:SIRCm limit slow subsystem} does not depend on $g$: the
choice of feedback mechanism disappears in the limit $\eps\to\infty$.
We can further see that
$\frac{d}{dt}(S+I) = \mu\left(1-(S+I)\right)$,
so the set $\{S+I=1\}$ is invariant and globally attracting for
\eqref{eq:SIRCm limit slow subsystem} at rate $\mu$; on it,
\eqref{eq:SIRCm limit slow subsystem} reduces to the logistic equation
$\dot I = \beta I(1-I)-(\mu+\alpha)I$.

\begin{rem}[Recovery of the classical SIS model]
Equations \eqref{eq:SIRCm limit slow subsystem} are, in fact, the classical SIS model, where recovered individuals return directly to $S$ with no intermediate immune class.
\end{rem}

The unique endemic equilibrium of \eqref{eq:SIRCm limit slow subsystem} is
\begin{align}
\label{eq:SIRCm limit fast subsystem}
    S_\infty^*\coloneqq\frac{\mu+\alpha}{\beta}=\BRN^{-1}, \qquad
    I_\infty^* \coloneqq1-\BRN^{-1},
\end{align}
and the Jacobian of \eqref{eq:SIRCm limit slow subsystem} at $(S^*_\infty,\infEE[\infty])$ is 
\begin{align}
\label{eq: Jac slow}
    M\coloneqq\begin{pmatrix}
        -\mu-\beta \infEE[\infty] & \alpha-\beta S^*_\infty\\[2pt]
        \beta \infEE[\infty] & 0
    \end{pmatrix}
    =\begin{pmatrix}
        -\mu-\beta \infEE[\infty] & -\mu\\[2pt]
        \beta \infEE[\infty] & 0
    \end{pmatrix},
\end{align}
with eigenvalues
\begin{align}
\label{eq: eigenvals Jac slow}
 \left\{-\mu,\; -\beta \infEE[\infty]\right\}
    = \left\{-\mu,\; \mu+\alpha-\beta\right\}.
\end{align}
Thus the equilibrium \eqref{eq:SIRCm limit fast subsystem} of the reduced
subsystem is hyperbolic and asymptotically stable if and only if $\BRN>1$.

\subsubsection{Local dynamics}

In this section, we show that not only the SIS model is recovered in the limit $\eps\to\infty$,
but also that the dynamics of SIRCm approximate those of SIS smoothly in $\eps$.
Let $S^{-}\coloneqq\tfrac12\tfrac{\mu}{\mu+\beta}$,
$S^{+}\coloneqq\tfrac12\bigl(\BRN^{-1}+1\bigr)$,
$I^{-}\coloneqq\tfrac12\tfrac{\mu}{\mu+\alpha}\infEE[\infty]$ and
$I^{+}\coloneqq\tfrac12\bigl(\infEE[\infty]+1\bigr)$, so that,
from Lemma \ref{lemma:bounds EE}
we have $\mathcal{W}\subset\mathrm{int}(\mathcal{W}^+)$ (with $\mathcal{W}$
defined as in Lemma \ref{lemma:uniform scaling}),
where
\begin{align*}
   \mathcal{W}^+ \coloneqq \left[S^{-},S^{+}\right]\times\left[I^{-},I^{+}\right].
\end{align*}
Since $g$ depends only on $S$ and $I$, the critical manifold $\critman$ is a
graph over the $(S,I)$ variables, and we may cut out of it the compact, smooth
submanifold (with boundary)
\begin{align}
\label{eq: compact critical submanifold}
    \critsubman: = \left\{\xfast\in\critman : (S,I)\in\mathcal{W}^{+}\right\}.
\end{align}
By Lemma \ref{lemma:uniform scaling},
we have that $g\geq g_{\min}$ in $\critsubman$. 
Computing the Jacobian of $\ffast$ with
respect to $\rfast$ and $\cfast$ gives
\begin{align}
\label{eq:Jacobian ffast}
    \partial_{\rfast,\cfast} \ffast(\xfast,{\epsinv}) = \begin{pmatrix}
        -{\epsinv} \mu - \left({\epsinv} + g(\xfast)\right)\delta
          & (1-\sigma)\beta {\epsinv} I \\[4pt]
        \left({\epsinv} + g(\xfast)\right)\delta
          & -\beta {\epsinv} I - \mu {\epsinv} - \left({\epsinv} + g(\xfast)\right)\gamma
    \end{pmatrix},
\end{align}
so that, for $\xfast\in\critsubman$, the eigenvalues of
$\partial_{\rfast,\cfast}\ffast(\xfast,0)$ are
\begin{align}
\label{eq: layer eigenvalues}
    \left\{-g(\xfast)\delta,\ -g(\xfast)\gamma\right\}.
\end{align}
Both are bounded away from $0$ uniformly in $\critsubman$, so $\critsubman$ is
an attracting normally hyperbolic submanifold of $\critman$.

Thus, Fenichel's Theorem \citet[Thm.~3.1.4]{kuehn2015} guarantees the 
existence of a locally invariant \emph{slow manifold}
\begin{gather}
\label{def: Fenichel manifold}
    \critsubman[\eps]\coloneqq\left\{\xfast \in \real^4:\;
    \bar{R} = \frac{\alpha I}{\delta g(\xfast)} + \bigo(\epsinv), \quad 
    \bar{C}
    =\frac{\alpha I}{\gamma g(\xfast)}+\bigo(\epsinv), 
    \quad (S,I)\in\mathcal{W}^{+}
    \right\},
\end{gather}
which is
diffeomorphic to $\critsubman$, and at a Hausdorff
distance $\mathcal{O}(\epsinv)$ from it.
Moreover, the flow in $\critsubman[\eps]$ converges 
(in $\mathcal{C}^1$) to
the flow in $\critsubman$ as $\eps\to\infty$.
In particular the flow in $\critsubman[\eps]$ is given by
\begin{align}
\label{eq:reduced flow Fenichel}
    \dot I = I \left[\beta S -(\mu+\alpha)\right]= 
    \beta I\left[(\infEE[\infty]-I)- \epsinv\left(\rfast+(1-\sigma)\cfast\right)\right].
\end{align}

\begin{prop}
\label{prop: large epsilon unique endemic}
    Let $\BRN>1$ and let $g$ satisfy Definition \ref{dfn:funcg}.
    Then, there exists an ${{\eps_\infty}}>0$ such that, for $\eps>\eps_\infty$, SIRCm \eqref{eq:SIRCm} has 
    a unique endemic equilibrium $\EE$,
    which depends smoothly on $\epsinv$  and is locally asymptotically stable.
    Moreover, we have that $\lim_{\eps\to\infty}\xfast^*_\eps=\xfast^*_\infty$,
    where
    \begin{align*}
    \EEfast[\infty] \coloneqq \left(\BRN^{-1},\infEE[\infty], \frac{\alpha \infEE[\infty]}{\delta g(S^*_\infty,\infEE[\infty],0,0)},
    \frac{\alpha \infEE[\infty]}{\gamma g(S^*_\infty,\infEE[\infty],0,0)}\right).
    \end{align*}
\end{prop}
\begin{proof}
\emph{Local uniqueness.}
Given $\xfast\in\real^4$ and $\eps>0$, 
we define $H(\xfast,\epsinv)\coloneqq\begin{pmatrix}
    \fslow(\xfast,\epsinv)\\
    \ffast(\xfast,\epsinv)
\end{pmatrix}$
and
\begin{align*}
A_\eps(\xfast)\coloneqq\partial_{S,I}\fslow,\quad
B_\eps(\xfast)\coloneqq\partial_{\rfast,\cfast}\fslow,\quad
D_\eps(\xfast)\coloneqq\partial_{S,I}\ffast,\quad
E_\eps(\xfast)\coloneqq\partial_{\rfast,\cfast}\ffast,
\end{align*}
with every term evaluated at $(\xfast,\epsinv)$.
It is straightforward to verify that $H(\EEfast[\infty],0)=0$.
To show that $\partial_{\xfast} H(\EEfast[\infty],0)$
is locally invertible, we compute its determinant  using the Schur complement
\begin{align*}
     \det \partial_{\xfast} H(\EEfast[\infty],0) = 
   \left[\det \left(E_\infty\right)\, \cdot\, 
    \det \left(A_\infty - B_\infty E_\infty^{-1} D_\infty \right)\right](\EEfast[\infty]),
\end{align*}
where $A_\infty$, $B_\infty$, $D_\infty$, $E_\infty$
 denote the corresponding limits as $\varepsilon\to\infty$.
From \eqref{eq:Jacobian ffast} and \eqref{eq: Jac slow} we have 
\begin{subequations}
\label{eq:limit Jacobians}
\begin{gather}
    E_\infty^{-1} (\EEfast[\infty])=
    \begin{pmatrix}
        -\frac{1}{g(\EEfast[\infty])\delta} & 0\\
        -\frac{1}{g(\EEfast[\infty])\gamma} & -\frac{1}{g(\EEfast[\infty])\gamma}
    \end{pmatrix},
    \\
    B_\infty E_\infty^{-1}  (\EEfast[\infty])=
    \begin{pmatrix}
        -1 & -1\\
        0 & 0
    \end{pmatrix},
    \\
    B_\infty E_\infty^{-1} D_\infty(\EEfast[\infty])=
    \begin{pmatrix}
        \frac{\alpha \infEE[\infty]}{g(\EEfast[\infty])}\partial_S g(\EEfast[\infty]) & -\alpha +\frac{\alpha \infEE[\infty]}{g(\EEfast[\infty])}\partial_I g(\EEfast[\infty])\\
        0 & 0
    \end{pmatrix},
    \\
    \left[ A_\infty - B_\infty E_\infty^{-1} D_\infty \right](\EEfast[\infty])=
    \begin{pmatrix}
        -\mu-\beta \infEE[\infty] & \alpha -\beta S^*_{\infty}\\
        \beta \infEE[\infty] & 0
    \end{pmatrix}=M,
\end{gather}
\end{subequations}
and therefore,
\begin{align*}
    \det \partial_{\xfast} H(\EEfast[\infty],0) = g(\EEfast[\infty])^2\delta\gamma \mu \beta \infEE[\infty]>0.
\end{align*}
By the IFT, there exists $\hat{\eps}>0$, a neighborhood of $\EEfast [\infty]$, $U_\infty\subset\real^4$,
and a unique smooth mapping $\bar \phi^*:[0, \hat{\eps}^{-1}] \to U_\infty$ 
such that $\bar\phi^*(0)=\EEfast[\infty]$ and $H(\bar \phi^*(\epsinv),\epsinv)=0$
for $\eps>\hat{\eps}$. 
Taking $\hat{\eps}$ larger if necessary,
$\phi^*(\epsinv)\in \Om$ with positive coordinates
for $\eps>\hat{\eps}$, so it is an endemic equilibrium. 
Recalling the notation from Remark \ref{rem: scaling notation}, the map $\phi^*$ is induced by $\bar{\phi}^*$ via the inverse 
scaling $\phi^*(\kappa) \coloneqq \operatorname{diag}(1, 1, \kappa, \kappa) \bar{\phi}^*(\kappa)$ for $\kappa \ge 0$.

\emph{Global uniqueness.} 
For the sake of contradiction, let us assume that there is no $\eps_\infty>0$ such that
$\phi^*(\epsinv)$ is
a unique endemic equilibrium in $\Om$ for $\eps>\eps_\infty$.
Then, there is a sequence $\EE[\eps_n]\in \Om$ 
such that
$\eps_n\to \infty$ as $n\to\infty$ and
$\EEfast[\eps_n]\neq\bar\phi^*(\eps_n^{-1})$ for all
$n$ large enough.
By Lemma \ref{lemma:uniform scaling} and Remark \ref{rem: scaling notation}, $\EEfast[\eps_n]\in \mathcal{B}$,
which is compact.
Using Bolzano-Weierstrass, there is a subsequence (without relabeling for simplicity) converging to some $\EEfast[]\in \mathcal{B}$.
Since $H$ is continuous on
$\mathcal{B}\times[0,\hat{\eps}^{-1})$,
\begin{align*}
        H(\xfast^{\,*},0)=\lim_{n\to\infty}H(\EEfast[\eps_n],\epsinv_n)=0.
\end{align*}
As a consequence, $\EEfast[]\in \critsubman$ and $\infEE[]=1-\BRN^{-1}$, $S^*=\BRN^{-1}$,
so $\EEfast[]=\EEfast[\infty]$.
Hence, for $n$ large enough, $\EEfast[\eps_n]\in U_\infty$ and $\eps_n>\hat{\eps}$. 
By the local uniqueness established above, we get a contradiction.

\emph{Local asymptotic stability.}
For each $\eps>0$,
the scaling $\xfast=(S,I,\eps R,\eps C)$ is a linear change of
variables, so the Jacobian of \eqref{eq:SIRCm} at $\EE$ is similar to
the Jacobian of the scaled system
\begin{align}
\label{eq: fast-slow writing}
\begin{pmatrix} \dot S , \dot I \end{pmatrix}^\top
   = \fslow(\xfast,\epsinv),
\qquad
\begin{pmatrix} \dot{\rfast} , \dot{\cfast} \end{pmatrix}^\top
   = \eps\,\ffast(\xfast,\epsinv),
\end{align}
evaluated at $\EEfast=\bar\phi^*(\epsinv)$, i.e., they share the same eigenvalues.
Then it suffices to show that, for $\eps$ large enough, the eigenvalues of
\begin{align*}
J_\eps(\bar\phi^*(\epsinv))=\begin{pmatrix} A_\eps & B_\eps \\ \eps D_\eps & \eps E_\eps
\end{pmatrix}(\bar\phi^*(\epsinv))
\end{align*}
have negative real parts.
Since $\bar\phi^*$ is a smooth function of $\epsinv$, the Jacobian blocks depend continuously
 on $\epsinv$ and approach their limiting values $A_\infty$, $B_\infty$, $D_\infty$, $E_\infty$
 displayed in \eqref{eq:limit Jacobians} as $\varepsilon\to\infty$.
 In particular, we know (see \eqref{eq: layer eigenvalues}) that the eigenvalues of $E_\infty$ are given by
 $-g_\infty \delta$ and $-g_\infty \gamma$, where $g_\infty \coloneqq g(\EEfast[\infty])\geq g_{\min}>0$.
 
One can recognize that 
\eqref{eq: fast-slow writing} linearized around 
$\bar\phi^*(\epsinv)$
is the singularly perturbed linear system \citet[Ch.~2, Eq. (2.1)--(2.2)]{kokotovic1999singular}
with small parameter $\epsinv$.
Moreover, we have that the blocks $A_\eps,B_\eps,D_\eps,E_\eps$
are continuously differentiable with respect to $\epsinv$ as $\epsinv\to0$, 
$E_\infty$ is non-singular with eigenvalues $-g_\infty\delta$ and $-g_\infty\gamma$,
and $A_\infty-B_\infty E_\infty^{-1}D_\infty=M$ has eigenvalues $-\mu$ and
$-\beta\infEE[\infty]$. 
Therefore, the conditions of \citet[Ch.~2, Thm.~3.1]{kokotovic1999singular} are satisfied
and the eigenvalues of $J_\eps(\bar\phi^*(\epsinv))$ are given by
\[
  \{-\mu+\mathcal{O}(\epsinv),\,
-\beta\infEE[\infty]+\mathcal{O}(\epsinv),\,
-\eps g_\infty\delta+\mathcal{O}(1),\,
-\eps g_\infty\gamma+\mathcal{O}(1)\}.
\]
Then,  we can enlarge $\eps_\infty$ so that every eigenvalue of the Jacobian of \eqref{eq:SIRCm} at $\EE$ has negative
real part for $\eps>\eps_\infty$; hence $\EE$ is locally asymptotically stable.
\end{proof}

The following statement follows immediately from Lemma \ref{lemma:bounds EE} and Proposition \ref{prop: large epsilon unique endemic}.
\begin{cor}
    $\infEE[\infty]$ is a supremum for $\infEE$.
\end{cor}

\subsubsection{Expansion of $\infEE$ for large $\eps$}\label{sec:high_prev_series}
By Proposition \ref{prop: large epsilon unique endemic},
it is possible to write the series expansion
\begin{align*}
\infEE= \infEE[\infty]+q(\eps^{-1})=\infEE[\infty]+\sum_{k\in\naturals^+} q_k\, \eps^{-k},
\end{align*}
for $\eps$ large enough.
We observe that, for every $\eps$, the endemic equilibrium satisfies
$
\eps^{-2}\poly_0(I)+\eps^{-1} \poly_1(I) +\poly_2(I) =0,
$
where $\poly_2(I)$ is a cubic polynomial in $I$ whose only non-zero root is $I_\infty^*\in(0,1)$.
Then, we can proceed analogously to Section \ref{sec:low_prev_series} 
to compute the coefficients $q_k$.
Approximations of up to order 4 in $\epsinv$ are depicted
in the right panels of Figures \ref{fig:eps1_eps2_prevalence} and 
\ref{fig:eps1_eps2_prevalence_I}.

\section{Seasonally forced dynamics}
\label{sec:seasonally_forced}

In the case of influenza, infections follow a strong seasonal variation.  In \cite{casagrandi2006}, the SIRC model was extensively studied in the presence of seasonal forcing.  As we have seen, seasonal forcing is not generally needed in order for the SIRCm model to produce long-term cyclically varying solutions.  Nevertheless, here we briefly present some examples of solutions obtained with the SIRCm model in the presence of seasonal forcing.

We restrict 
our attention to the prevalence-driven model with a single feedback parameter ($\varepsilon\coloneqq\varepsilon_1 = \varepsilon_2$),  since (like the SIRC model) this model's unforced endemic equilibrium remains stable for all $\epstilde$ tested (see Section~\ref{sec: stability and bifurcation}).  This lets us isolate the effect of mutation feedback on the seasonally forced dynamics.
For the forcing, we replace the constant transmission rate $\beta_0$ with $\beta(t) = \beta_0\left(1 + \eta\cos(2\pi t)\right)$,
as is done in \cite{casagrandi2006}.

\begin{figure}
    \centering \includegraphics[width=1\textwidth]{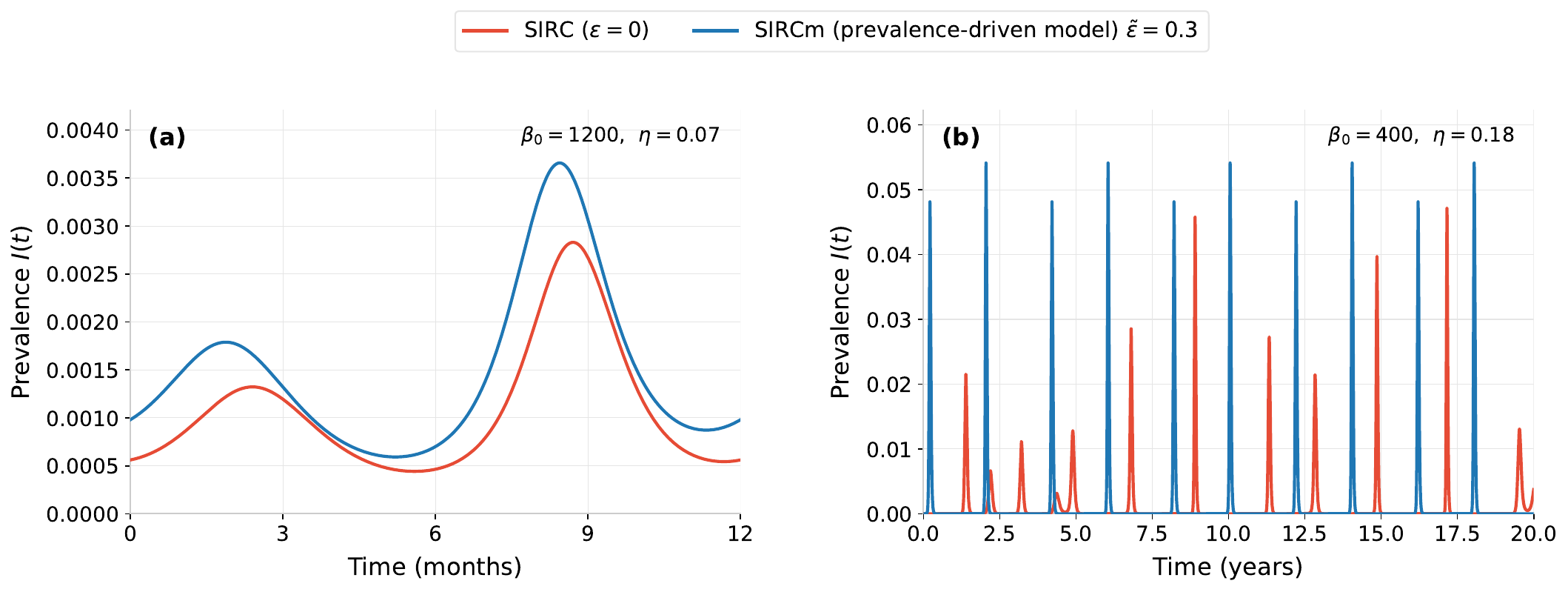}
    \caption{(see Section \ref{sec:seasonally_forced}) Replication and extension of Figure 5 in \cite{casagrandi2006} for the prevalence-driven model of the SIRCm model ($\tilde{\varepsilon}=0.3$) compared to the standard SIRC baseline ($\tilde{\varepsilon}=0$). The parameter $\eta$ represents the degree of seasonality
    }
    \label{fig:seasonallyforced}
\end{figure}

Figure~\ref{fig:seasonallyforced} shows a comparison of the SIRCm prevalence-driven model ($\tilde\varepsilon = 0.3$) with the SIRC model ($\tilde\varepsilon = 0$) with parameters corresponding to \citet[Figure 5]{casagrandi2006}: panel (a) uses parameters representative of a tropical regime with low seasonality and high baseline transmission, while panel (b) uses those representative of a temperate regime with stronger seasonal forcing and lower baseline transmission.

In the tropical regime (a), both models settle into a stable annual cycle, with SIRCm attaining a higher peak prevalence than SIRC, consistent with the elevated endemic prevalence documented in Section~\ref{sec: stability and bifurcation}. In the temperate regime (b), while SIRC settles into the chaotic attractor reported by Casagrandi et al.\ for this parameter setting, SIRCm  converges to a period-2 cycle. These two regimes illustrate that mutation feedback can alter not only the level but also the qualitative character of the seasonally forced dynamics; further comparisons with different values of $\epstilde$ are given in Appendix~\ref{app:seasonal}.

\section{Discussion and biological implications}
\label{sec:discussion}

The SIRCm model provides a simple, implicit way to incorporate the effect of viral immune escape into a compartmental epidemiological model.  It smoothly bridges between different regimes; when mutation feedback is weak or non-existent, its behavior is similar to that of the SIRC model
\citep{casagrandi2006}, in which recovered hosts pass through a period of partial immunity (cross-immunity) before
eventually becoming fully susceptible again.
In this regime, the endemic equilibrium prevalence rate is typically quite low.
On the other hand, with strong enough mutation feedback, the endemic prevalence becomes high (more than 50\% of the population is typically infected), as the rate of immune escape and high prevalence together 
become self-sustaining: greater incidence leads to more rapid immune escape, which leads to greater incidence.
In the limit of strong feedback, 
the model reduces to the classical SIS
model of a disease that confers no immunity at all.

The transition between regimes is not gradual. Figures \ref{fig:eps1_eps2_prevalence} and \ref{fig:eps1_eps2_prevalence_I} show that the long-run fraction of the population that is infected at any given time, $I^*$, stays close to its SIRC value across a wide range of weak-to-moderate mutation feedback, and rises sharply for feedback strength $\epstilde$ above a certain threshold, beyond which $I^*$ climbs quickly toward its theoretical ceiling. 

Given this behavior, it is important to consider the factors that influence the mutation feedback strength, in order to understand what could drive a transition to high prevalence.  Those factors include not only the rate of reproduction of the pathogen and the probability of a beneficial mutation, but also the overall size of the host population -- with enough hosts, immune escape can potentially proceed at a constantly high rate.  Of course, the model's idealization of immune escape as a linear function of incidence or prevalence is too simplistic, and it is possible that some sort of saturation occurs in very large host populations.

Besides this transition between low- and high-prevalence endemic regimes,
the results of Section \ref{sec: stability and bifurcation} show that mutation feedback changes the qualitative behavior of epidemics, producing (in the 
transmission-driven
model) sustained oscillations.
Interestingly, the instability that leads to these
oscillations occurs only in an intermediate mutation feedback regime.
This is notable because nothing in the model changes with time; there is no built-in seasonality and no external driver of any kind. The oscillation is generated entirely by the feedback loop between infection and mutation. This offers a  different explanation for recurring epidemic waves than the usual one \citep{Martinez2018}, which attributes them to external seasonal changes in transmission (e.g., people spending more time indoors in the winter).  From this viewpoint, seasonal behavioral patterns may serve primarily to reinforce the timing of oscillations, rather than as their primary driver.

\section*{Statements and Declarations}

\noindent \textbf{Artificial intelligence usage.} 
The use of AI language models (Anthropic's Claude, Google's Gemini, OpenAI's ChatGPT) is acknowledged. All content, proofs and figures were produced and verified by the authors, who take full responsibility for the final work. 

\noindent \textbf{Competing Interests.} 
The authors have no competing interests that are relevant to the content of this article.

\noindent \textbf{Code availability.} 
The Python, Julia and \textit{Mathematica} code used to generate the numerical results and figures in this study is available at \url{https://github.com/danielarenee/infection-modeling-with-mutation}.

\section*{Acknowledgements} 
This work was supported by funding from King Abdullah University of Science and Technology (KAUST), including support through the Visiting Student Research Program (VSRP). 

\begin{appendices}
\counterwithout{figure}{section}
\setcounter{figure}{15}

\section{The equilibrium polynomial for the transmission-driven model}
\label{app:quartic}

As stated in Section \ref{sec: persistence theory}, eliminating $S,R,C$ from the equilibrium equations of SIRCm \eqref{eq:SIRCm} yields the following polynomial in $I$ 

\begin{align*}
    \poly(\infEE[],\eps_1,\eps_2)= c_4 {\infEE[]}^4 + c_3 {\infEE[]}^3+ c_2 {\infEE[]}^2 +c_1 {\infEE[]} +c_0 = 0, 
\end{align*}
where
\begin{align*}
    c_4 &= -\beta^{2} \gamma \delta \mu \eps_1 \eps_2 (-\alpha^{2} \gamma \eps_2
            +\alpha (\beta^{2} \sigma 
            +\gamma \delta \sigma \eps_1-\gamma \eps_2 (\delta \sigma +2 \mu ))\\
        &\qquad +\beta \sigma (\beta (\delta  \sigma +\mu )+\delta \mu \sigma \eps_1)
            +\gamma \mu (\delta  \sigma \eps_1-\eps_2 (\delta \sigma +\mu ))),\\
    c_3 &= \beta \mu (\alpha^{3} \gamma^{2} \delta  \eps_1 \eps_2^2 
            +\alpha^{2} \gamma \eps_2 (\beta^{2}(\gamma \eps_2 +\delta (\eps_1-2 \sigma \eps_1))-\beta \gamma \delta \eps_1 \eps_2
            +\gamma \delta \eps_1 (4 \mu \eps_2-\delta (\sigma +1)(\eps_1-\eps_2)))\\
        &\qquad +\alpha (\beta^{4} (\sigma -1) (\delta \sigma \eps_1-\gamma \eps_2)
            +\beta^{3} \gamma \delta \sigma \eps_1 \eps_2 +\beta^{2} \gamma (\gamma \eps_2 (\delta  (-2 \sigma  \eps_1+\eps_1+\sigma \eps_2+\eps_2)+2 \mu  \eps_2)\\
        &\qquad +\delta \eps_1 (\delta (\sigma -1) \sigma \eps_1-2 \eps_2 (\delta \sigma^{2}
            +3 \mu \sigma -\mu )))+\beta  \gamma \delta \eps_1 \eps_2 (2 \gamma  \delta \sigma \eps_1-\gamma \eps_2 (2 \delta \sigma 
            +\mu (\sigma +2))\\
        &\qquad -\delta \mu (\sigma +1) \sigma  \eps_1)+\gamma^{2} \delta \mu \eps_1 \eps_2 (5 \mu \eps_2-2 \delta (\sigma +1)(\eps_1-\eps_2)))
            +\beta^{4} (\sigma -1) (\delta \sigma +\mu) (\delta \sigma \eps_1-\gamma \eps_2)
            \\
        &\qquad +\beta^{3} \delta  \sigma  \eps_1 (\gamma  \eps_2 (2 \delta  \sigma -\mu  (\sigma -3)) +\delta \mu (\sigma -1)\sigma \eps_1)
            +\beta^{2} \gamma (\gamma \eps_2 (\delta^{2} \sigma (\eps_2
            -\eps_1)\\
        &\qquad +\delta \mu (-2 \sigma \eps_1+\eps_1+\sigma \eps_2+\eps_2)
            +\mu^{2} \eps_2) +\delta \mu \eps_1 (\delta(\sigma -1)\sigma \eps_1 +\eps_2 (\delta \sigma^{2}(2 \eps_1-3)-4 \mu \sigma 
            +\mu )))\\
        &\qquad +\beta \gamma \delta \mu \eps_1 \eps_2 
            (2 \gamma \delta \sigma \eps_1-\gamma \eps_2(2 \delta \sigma 
            +\mu \sigma +\mu ) -\delta \mu (\sigma +1) \sigma \eps_1)\\
         &\qquad +\gamma^{2}\delta \mu^{2}\eps_1 \eps_2 (2 \mu  \eps_2-\delta (\sigma +1) (\eps_1-\eps_2))), \\
    c_2 &= \mu  (\alpha^{3} \gamma^{2} \delta  \eps_1 \eps_2 (\delta  (\eps_2-\eps_1)+\mu  \eps_2)-\alpha^{2} \gamma  (\beta^{2} (\delta  \eps_2 (\gamma  \sigma  \eps_1-2 \gamma  \eps_2+4 \mu  \sigma  \eps_1-2 \mu  \eps_1)-2 \gamma  \mu  \eps_2^2\\
    &\qquad  +\delta^{2} \eps_1 (-\sigma  \eps_1+\eps_1+2 \sigma  \eps_2-\eps_2))+\beta  \delta  \eps_1 \eps_2 (\delta  \mu  \sigma  \eps_1-\gamma  (\sigma +1) (\delta  \eps_1-\delta  \eps_2-\mu  \eps_2))\\
    &\qquad  +3 \gamma  \delta  \mu  \eps_1 \eps_2 (\delta  (\eps_1-\eps_2)-\mu  \eps_2))+\alpha  (\beta^{4} (\sigma -1) (-\gamma^{2} \eps_2+\gamma  \delta  (\sigma  \eps_1+\eps_1-2 \eps_2)-3 \gamma  \mu  \eps_2\\
    &\qquad  +\delta  \sigma  \eps_1 (\delta +3 \mu )) +\beta^{3} (\gamma^{2} (-\eps_2) (\delta  (-2 \sigma  \eps_1+\eps_1+\sigma  \eps_2+\eps_2)+\mu  (\sigma +1) \eps_2)+\gamma  \delta  \eps_1 (-\delta  (\sigma -1) \sigma  \eps_1\\
    &\qquad +2 \delta  \sigma^{2} \eps_2 +\mu  (\sigma^{2}+3 \sigma -1) \eps_2)+\delta^{2} \mu  (\sigma -1) \sigma  \eps_1^2)+\beta^{2} \gamma  (\gamma  (-(\delta^{2} (\eps_1-\eps_2) (\sigma  \eps_1 (\eps_2-1)+\eps_1\\
    &\qquad +(\sigma +1) \eps_2)) +\delta  \mu  \eps_2 (\sigma  \eps_1 (\eps_2-4)+\eps_1+(\sigma +5) \eps_2)+4 \mu^{2} \eps_2^2)+\delta  \mu  \eps_1 (\delta  \eps_1 (\sigma^{2} (\eps_2+1)+\sigma  (\eps_2+1)\\
    &\qquad -2) -2 \delta  (\sigma^{2}+2 \sigma -1) \eps_2+\mu  (4-9 \sigma ) \eps_2))+2 \beta  \gamma  \delta  \mu  \eps_1 \eps_2 (\gamma  (\sigma +1) (\delta  (\eps_1-\eps_2)-\mu  \eps_2)-\delta  \mu  \sigma  \eps_1)\\
    &\qquad  +3 \gamma^{2} \delta  \mu^{2} \eps_1 \eps_2 (\delta  (\eps_2-\eps_1)+\mu  \eps_2))+\beta^{5} (\sigma -1) (\delta  \sigma +\mu ) (\gamma  \eps_2-\delta  \sigma  \eps_1)-\beta^{4} (\gamma^{2} (\sigma -1) \eps_2 (\delta +\mu )\\
    &\qquad +\gamma  (\delta^{2} \sigma  (\eps_1 (\sigma  (\eps_2-2)+2)+(\sigma -1) \eps_2)+\delta  \mu  (2 (\sigma^{2}-1) \eps_2-\eps_1 (\sigma^{2}+(\sigma -2) \sigma  \eps_2-1))\\
    &\qquad +3 \mu^{2} (\sigma -1) \eps_2) +\delta  \mu  (\sigma -1) \sigma  \eps_1 (\delta  (\sigma  (\eps_1-2)-1)-3 \mu ))+\beta^{3} (\gamma^{2} (-\eps_2) (2 \delta^{2} \sigma  (\eps_2-\eps_1)\\
    &\qquad  +\delta  \mu  (-3 \sigma  \eps_1+\eps_1+3 \sigma  \eps_2+\eps_2) +\mu^{2} (\sigma +1) \eps_2)-\gamma  \delta  \mu  \eps_1 \eps_2 (\delta  \sigma^{2} (\eps_1-4)-5 \mu  \sigma +\mu )\\
    &\qquad  +\delta^{2} \mu^{2} \sigma  (\sigma^{2}-1) \eps_1^2) +\beta^{2} \gamma  \mu  (\gamma  (-(\delta^{2} (\eps_1-\eps_2) (\sigma  \eps_1 (\eps_2-1)+\eps_1+(\sigma +1) \eps_2))+\delta  \mu  \eps_2 (\sigma  \eps_1 (\eps_2-3)\\
    &\qquad +\eps_1+(\sigma +3) \eps_2) +2 \mu^{2} \eps_2^2)+\delta  \mu  \eps_1 (\delta  ((\sigma +1) \eps_1 (\sigma +\sigma  \eps_2-1)-2 \sigma^{2} \eps_2-2 \sigma  \eps_2+\eps_2)\\
    &\qquad +\mu  (2-5 \sigma ) \eps_2)) +\beta  \gamma  \delta  \mu^{2} \eps_1 \eps_2 (\gamma  (\sigma +1) (\delta  (\eps_1-\eps_2)-\mu  \eps_2)-\delta  \mu  \sigma  \eps_1)+\gamma^{2} \delta  \mu^{3} \eps_1 \eps_2 (\delta  (\eps_2-\eps_1)+\mu  \eps_2)), \\
    c_1 &= \beta \mu (\alpha^{2} \gamma (\gamma (\delta (\eps_1-\eps_2)-\mu \eps_2) 
            (\mu (-\eps_2)-\delta (-\sigma \eps_1+\eps_1+\eps_2))+\delta \mu \eps_1 (\delta ((\sigma -1) \eps_1-2 \sigma \eps_2+\eps_2)\\
        &\qquad +\mu (1-2 \sigma) \eps_2))+\alpha (\beta^{2} (\sigma -1) (\gamma^{2} (\delta
            (\eps_1-2 \eps_2) -2 \mu\eps_2) +\gamma (\delta^{2} (\sigma \eps_1+\eps_1-\eps_2)\\
        &\qquad +2 \delta \mu (\sigma  \eps_1+\eps_1-2 \eps_2)-3 \mu^{2} \eps_2) +\delta \mu \sigma \eps_1 (2 \delta +3 \mu ))+\beta 
            (-(\gamma^{2} (\delta (\eps_1-\eps_2)-\mu \eps_2)(\delta (\sigma -1) \eps_1\\
        &\qquad -\delta (\sigma +1) \eps_2-\mu (\sigma +1) \eps_2)) +\gamma \delta \mu \eps_1 (\delta  (\sigma -1)^2 \eps_1 +\delta (\sigma^{2}+2 
            \sigma -1) \eps_2+\mu (\sigma^{2}+2 \sigma -1) \eps_2)\\
        &\qquad +\delta^{2} \mu^{2} (\sigma -1) \sigma \eps_1^2) +2 \gamma \mu (\gamma  (\delta (\eps_1-\eps_2)-\mu \eps_2) (\mu (-\eps_2)-\delta 
            (-\sigma  \eps_1+\eps_1+\eps_2))+\delta \mu \eps_1 (\delta ((\sigma -1)\eps_1\\
        &\qquad -2 \sigma \eps_2+\eps_2) +\mu (1-2 \sigma) \eps_2)))+\beta^{3} (\sigma -1) (\gamma^{2} \eps_2 (\delta +\mu )+\gamma (\delta^{2} \sigma  
            (\eps_2-2 \eps_1)+\delta \mu ((\sigma +2) \eps_2\\
        &\qquad -(\sigma +1) \eps_1)+2 \mu^{2} \eps_2) -\delta \mu \sigma \eps_1 (\delta \sigma +\delta +2 \mu ))-\beta^{2} (\gamma^{2} (\delta +\mu ) (\delta (\eps_1-\eps_2) (\sigma (\eps_2-1)+1)\\
        &\qquad -\mu \eps_2 (\sigma (\eps_2-2)+2)) +\gamma \mu (\delta^{2} ((\sigma -1)\sigma  \eps_1^2+2 \sigma \eps_1+\sigma^{2} \eps_1 (\eps_2-3)+\eps_1+(\sigma^{2}-1) \eps_2)\\
        &\qquad +\delta \mu (\eps_1 (\sigma^{2} 
            (\eps_2-2)+2) +(\sigma^{2}+3 \sigma -4) \eps_2)+3 \mu^{2} (\sigma -1) \eps_2)+\delta \mu^{2} (\sigma -1) \sigma \eps_1 (\delta (-\sigma +\eps_1-2)\\
        &\qquad -3 \mu ))+\beta \mu (-(\gamma^{2} (\delta (\eps_1-\eps_2)-\mu \eps_2) (\delta (\sigma -1) \eps_1-\delta (\sigma +1) \eps_2  -\mu (\sigma +1) \eps_2))+\gamma \delta \mu \eps_1 (\delta(\sigma -1)^2 
            \eps_1\\
        &\qquad +\delta (\sigma^{2}+2 \sigma -1) \eps_2+\mu (\sigma^{2}+2 \sigma -1) \eps_2)+\delta^{2} \mu^{2} (\sigma -1) \sigma \eps_1^2) +\gamma \mu^{2} (\gamma (\delta (\eps_1-\eps_2)-\mu \eps_2) (\mu (-\eps_2)\\
        &\qquad -\delta (-\sigma  \eps_1+\eps_1+\eps_2))+\delta  \mu  \eps_1 (\delta  ((\sigma -1) \eps_1-2 \sigma  \eps_2+\eps_2)+\mu  (1-2 \sigma ) \eps_2))), \\
    c_0 &= \beta^{2} \mu  (\sigma -1) (\gamma +\mu ) (\delta +\mu ) (\alpha -\beta +\mu ) (\gamma  \delta (\eps_1-\eps_2)- \gamma  \mu  \eps_2+\delta  \mu  \sigma \eps_1) .\\
\end{align*}

\noindent When setting $\eps_1 = \eps_2 = \eps$, we get a simpler 4 degree polynomial, analogous to the single-$\eps$ formulation used in Remark \ref{cubicpolyremark} for the prevalence-driven model:

\begin{align*}
    \poly(\infEE,\eps)= \tilde{c}_4 {\infEE}^4 + \tilde{c}_3 {\infEE}^3+ \tilde{c}_2 {\infEE}^2 +\tilde{c}_1 I +\tilde{c}_0 = 0, 
\end{align*}
where

\begin{align*}
    \tilde{c}_4 &= \beta^{2} \gamma \delta  \mu  \eps  (-\alpha^{2} \gamma  \eps +\alpha (\beta^{2} \sigma -2 \gamma  \mu  \eps )+\beta \sigma  (\beta  \delta \sigma +\beta \mu +\delta \mu  \sigma \eps )-\gamma  \mu^{2} \eps ), \\
    \tilde{c}_3 &= \beta  \mu  (-\alpha^{3} \gamma^{2} \delta  \eps^{2}-\alpha^{2} \gamma  \eps  (\beta^{2} (\gamma -2 \delta  \sigma +\delta )-\beta  \gamma  \delta  \eps +4 \gamma  \delta  \mu  \eps )+\alpha  (\beta^{4} (\sigma -1) (\gamma -\delta  \sigma )-\beta^{3} \gamma  \delta  \sigma  \eps \\
    &\qquad + \beta^{2} \gamma  \eps  (\gamma  \delta  (\sigma -2)-2 \gamma  \mu +\delta^{2} \sigma  (\sigma +1)+\delta  \mu  (6 \sigma -2))+\beta  \gamma  \delta  \mu  \eps^{2} (\gamma  (\sigma +2)+\delta  \sigma  (\sigma +1))\\
    &\qquad -5 \gamma^{2} \delta  \mu^{2} \eps^{2})+\beta^{4} (\sigma -1) (\gamma -\delta  \sigma ) (\delta  \sigma +\mu )+\beta^{3} \delta  \sigma  \eps  (-2 \gamma  \delta  \sigma +\gamma  \mu  (\sigma -3)-\delta  \mu  (\sigma -1) \sigma )\\
    &\qquad +\beta^{2} \gamma  \mu  \eps  (\gamma  \delta  (\sigma -2)-\gamma  \mu +\delta^{2} \sigma  (1-2 \sigma  (\eps -1))+\delta  \mu  (4 \sigma -1))+\beta  \gamma  \delta  \mu^{2} (\sigma +1) \eps^{2} (\gamma +\delta  \sigma ))\\
    &\qquad -2 \gamma^{2} \delta  \mu^{3} \eps^{2}, \\
    \tilde{c}_2 &= \mu  (-\alpha^{3} \gamma^{2} \delta  \mu  \eps^{2}+\alpha^{2} \gamma  \eps  (\beta^{2} (\gamma  \delta  (\sigma -2)-2 \gamma  \mu +\delta^{2} \sigma +\delta  \mu  (4 \sigma -2))+\beta  \delta  \mu  \eps  (\gamma  \sigma +\gamma +\delta  \sigma )-3 \gamma  \delta  \mu^{2} \eps )\\
    &\qquad +\alpha  (\beta^{4} (\sigma -1) (\gamma +\delta +3 \mu ) (\gamma -\delta  \sigma )-\beta^{3} \eps  (\gamma^{2} (\delta  (\sigma -2)-\mu  (\sigma +1))+\gamma  \delta  (\delta  \sigma  (\sigma +1)\\
    &\qquad +\mu  (\sigma^{2}+3 \sigma -1)) +\delta^{2} \mu  (\sigma -1) \sigma )-\beta^{2} \gamma  \mu  \eps  (\gamma  \delta  (\sigma  (\eps -3)+6)+4 \gamma  \mu +\delta^{2} \sigma  (-\sigma +\sigma  \eps +\eps -3)\\
    &\qquad +\delta  \mu  (4-9 \sigma )) +2 \beta  \gamma  \delta  \mu^{2} \eps^{2} (\gamma  \sigma +\gamma +\delta  \sigma )-3 \gamma^{2} \delta  \mu^{3} \eps^{2})+\beta^{5} (\sigma -1) (\delta  \sigma -\gamma ) (\delta  \sigma +\mu )\\
    &\qquad +\beta^{4} (\gamma^{2} (\sigma -1) (\delta +\mu ) +\gamma  (\delta^{2} \sigma  (\sigma  (\eps -1)+1)-\delta  \mu  (-\sigma^{2}+(\sigma -2) \sigma  \eps +1)+3 \mu^{2} (\sigma -1))\\
    &\qquad +\delta  \mu  (\sigma -1) \sigma  (\delta  (\sigma  (\eps -2)-1)-3 \mu )) +\beta^{3} \mu  \eps  (\gamma^{2} (2 \delta +\mu  \sigma +\mu )+\gamma  \delta  (\delta  \sigma^{2} (\eps -4)-5 \mu  \sigma +\mu )\\
    &\qquad -\delta^{2} \mu  \sigma  (\sigma^{2}-1))-\beta^{2} \gamma  \mu^{2} \eps  (\gamma  \delta  (\sigma  (\eps -2)+4) +2 \gamma  \mu +\delta^{2} \sigma  (-\sigma +\sigma  \eps +\eps -2)+\delta  \mu  (2-5 \sigma ))\\
    &\qquad +\beta  \gamma  \delta  \mu^{3} \eps^{2} (\gamma  \sigma +\gamma +\delta  \sigma )-\gamma^{2} \delta  \mu ^4 \eps^{2}), \\
    \tilde{c}_1 &= - \beta  \mu  (\alpha^{2} \gamma  \mu  \eps  (\gamma  (\mu -\delta  (\sigma -2))+\delta  (-\delta  \sigma -2 \mu  \sigma +\mu ))-\alpha  (\beta^{2} (\sigma -1) (\gamma  (\delta +2 \mu )+\mu  (2 \delta +3 \mu )) (\gamma -\delta  \sigma )\\
    &\qquad +\beta  \mu  \eps  (\gamma^{2} (2 \delta +\mu  \sigma +\mu )-\gamma  \delta  (2 \delta  \sigma^{2}+\mu  (\sigma^{2}+2 \sigma -1))-\delta^{2} \mu  (\sigma -1) \sigma )+2 \gamma  \mu^{2} \eps  (\gamma  \delta  (\sigma -2)\\
    &\qquad -\gamma  \mu +\delta^{2} \sigma +\delta  \mu  (2 \sigma -1)))+\beta^{3} (\sigma -1) (\gamma -\delta  \sigma ) (\gamma  (\delta +\mu )+\mu  (\delta  \sigma +\delta +2 \mu ))\\
    &\qquad +\beta^{2} \mu  (\gamma^{2} (\delta +\mu ) (\sigma  (\eps -2)+2)+\gamma  (\delta^{2} \sigma  (2 \sigma -2 \sigma  \eps +\eps -2)+\delta  \mu  (-3 \sigma -(\sigma^{2} (\eps -1))+2)\\
    &\qquad -3 \mu^{2} (\sigma -1)) +\delta  \mu  (\sigma -1) \sigma  (\delta  (\sigma -\eps +2)+3 \mu ))-\beta  \mu^{2} \eps  (\gamma^{2} (2 \delta +\mu  \sigma +\mu )-\gamma  \delta  (2 \delta  \sigma^{2} \\
    &\qquad +\mu  (\sigma^{2}+2 \sigma -1)) -\delta^{2} \mu  (\sigma -1) \sigma )+\gamma  \mu^{3} \eps  (\gamma  (\mu -\delta  (\sigma -2))+\delta  (-\delta  \sigma -2 \mu  \sigma +\mu ))),\\
    \tilde{c}_0 &= \beta^{2} \mu^{2} (\sigma -1) (\gamma +\mu ) (\delta +\mu ) (\alpha -\beta +\mu ) (\gamma -\delta  \sigma ).\\
\end{align*}

\section{Stability of endemic equilibria for the prevalence-driven model}\label{appendixB}

Figure \ref{fig:eps1_eps2_prevalence_I} replicates the endemic-prevalence analysis for the prevalence-driven model, analogous to Figure \ref{fig:eps1_eps2_prevalence} for the transmission-driven case.

Section \ref{sec: stability and bifurcation} restricts attention to the diagonal $\epstilde_1 = \epstilde_2 = \epstilde$, along which the endemic equilibrium of the prevalence-driven model remains locally asymptotically stable for every $\beta_0$ in the realistic range tested and every $\eps \geq 0$ (Figure \ref{fig:eps1_eps2_hopf_surface_I}). This stability does not extend to the full two-parameter $(\epstilde_1, \epstilde_2)$ space; Figure \ref{fig:eps1_eps2_hopf_surface_I} shows an unstable region for sufficiently small $\epstilde_1$, lying just off the diagonal but not crossing it. 

\begin{figure}
    \centering
    \includegraphics[width=1\textwidth]{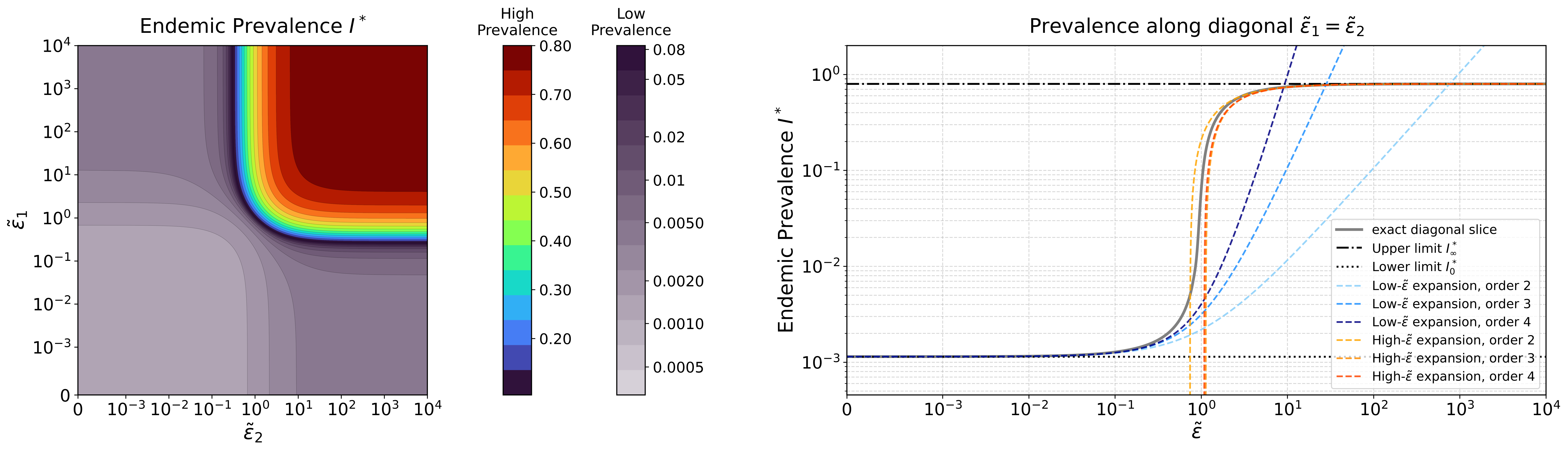}
    \caption{(see Appendix \ref{appendixB}) Endemic prevalence $I^*$ of the prevalence-driven model of the SIRCm model, analogous to Figure~\ref{fig:eps1_eps2_prevalence}. \emph{Left:} two-dimensional plot of $I^*$ as a function of $\varepsilon_1$ and $\varepsilon_2$. \emph{Right:} diagonal slice $\tilde\varepsilon_1 = \tilde\varepsilon_2$, together with the first three orders of the Taylor series approximations for the low- and high-feedback regimes. The dotted and dash-dotted lines mark the lower and upper limits $\infEE[0]$ and $I_\infty^*$, respectively, as defined in the caption of Figure~\ref{fig:eps1_eps2_prevalence}}    \label{fig:eps1_eps2_prevalence_I}
\end{figure}

Beyond prevalence, Remark~\ref{cubicpolyremark} notes that once the endemic prevalence $I^*_\eps$ is known, the remaining components of the endemic equilibrium of the prevalence-driven model can be computed. Writing $G := 1 + \varepsilon I^*_\varepsilon$, the equilibrium equations $\dot C = 0$ and $\dot R = 0$ give, in  
order,

\begin{align*}
  C^* &= \frac{\alpha\,\delta\,G\,I^*_\varepsilon}
              {\beta I^*_\varepsilon\,(\mu + \sigma\delta G) + (\mu + \delta G)(\mu + \gamma G)}, \\[4pt]
  R^* &= \frac{C^*\,(\beta I^*_\varepsilon + \mu + \gamma G)}{\delta G}, \\[4pt]
  S^* &= \mathcal{R}_0^{-1} - \sigma C^*,
\end{align*}

where the last identity is~\eqref{eq: relation SC endemic} obtained from $\dot I = 0$.

\section{Seasonally forced time series---some more examples}
\label{app:seasonal}

Figures \ref{fig:seasonalforcing-example}--\ref{fig20} provide additional examples of the seasonally forced comparison introduced in Section \ref{sec:seasonally_forced}. We contrast the tropical regime ($\beta_0 = 1200$, $\eta = 0.07$) against the temperate regime ($\beta_0 = 400$, $\eta = 0.18$), using the same SIRC baseline ($\tilde\varepsilon = 0$) for reference.

\begin{figure}
    \centering
    \includegraphics[width=1\textwidth]{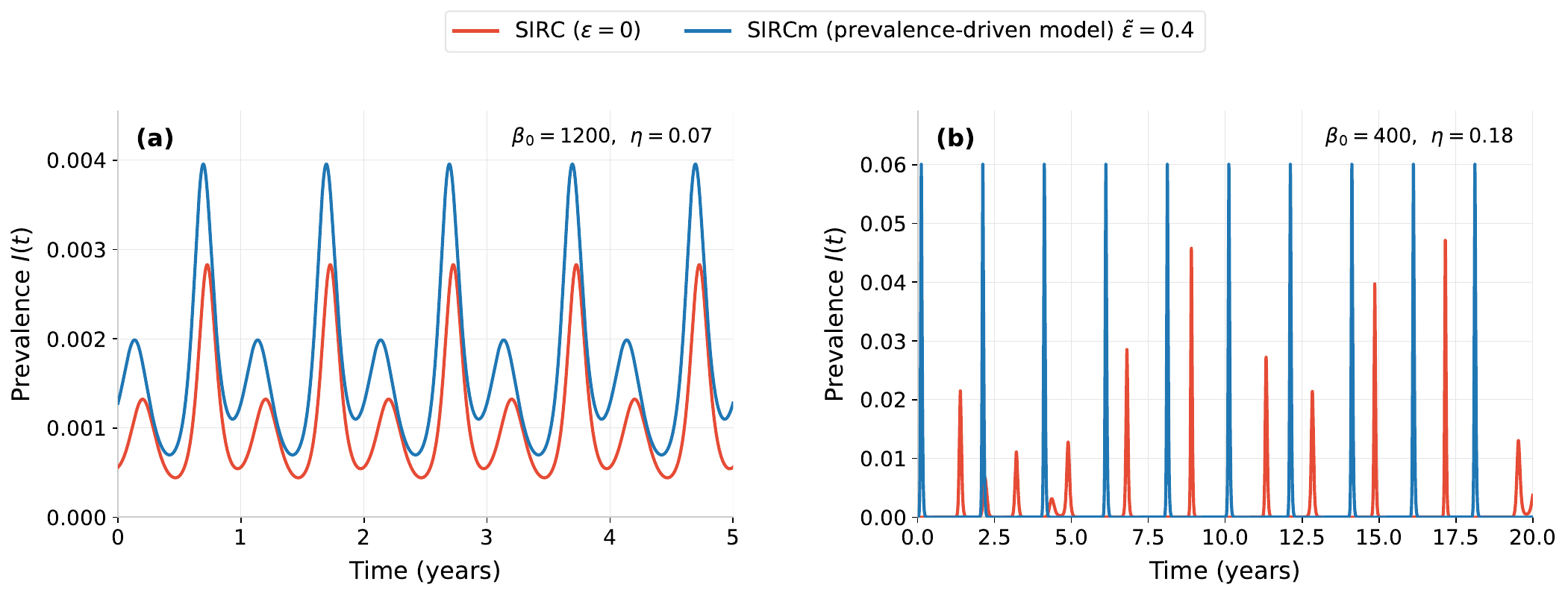}
    \caption{(see Appendix \ref{app:seasonal}) Replication and extension of Figure 5 in \cite{casagrandi2006} for the prevalence-driven model of the SIRCm model ($\tilde{\varepsilon}=0.4$) compared to the standard SIRC baseline ($\tilde{\varepsilon}=0$). In the tropical regime (a), both models remain on an annual cycle, with SIRCm attaining higher peak and trough prevalence than SIRC. In the temperate regime (b), SIRC settles into the chaotic attractor reported by Casagrandi et al., while SIRCm shows a regular periodic cycle}    
    \label{fig:seasonalforcing-example}
\end{figure}

\begin{figure}
    \centering
    \includegraphics[width=1\textwidth]{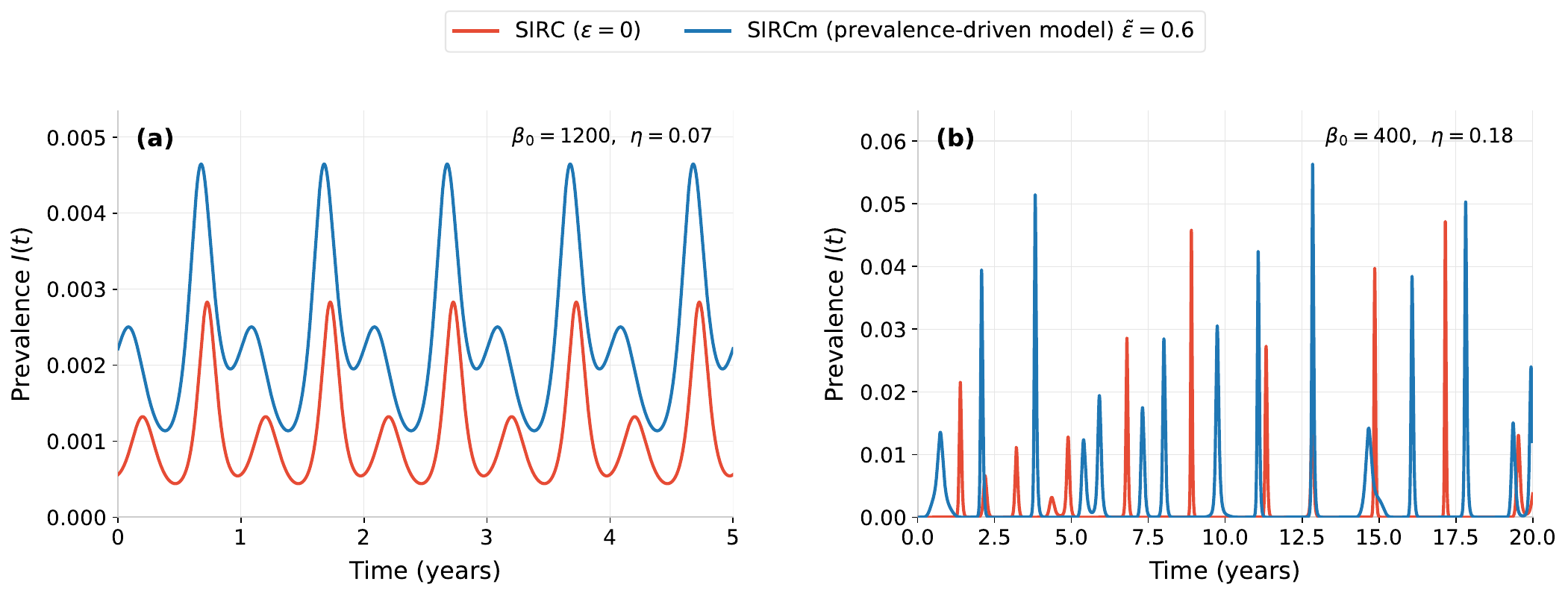}
    \caption{(see Appendix \ref{app:seasonal}) As Figure~\ref{fig:seasonalforcing-example}, at $\tilde\varepsilon = 0.6$. In the tropical regime (a), both models remain on an annual cycle, with SIRCm attaining higher peak and trough prevalence than SIRC. In the temperate regime (b), both models show chaotic behavior}    
    \label{fig18}
\end{figure}

\begin{figure}
    \centering
    \includegraphics[width=1\textwidth]{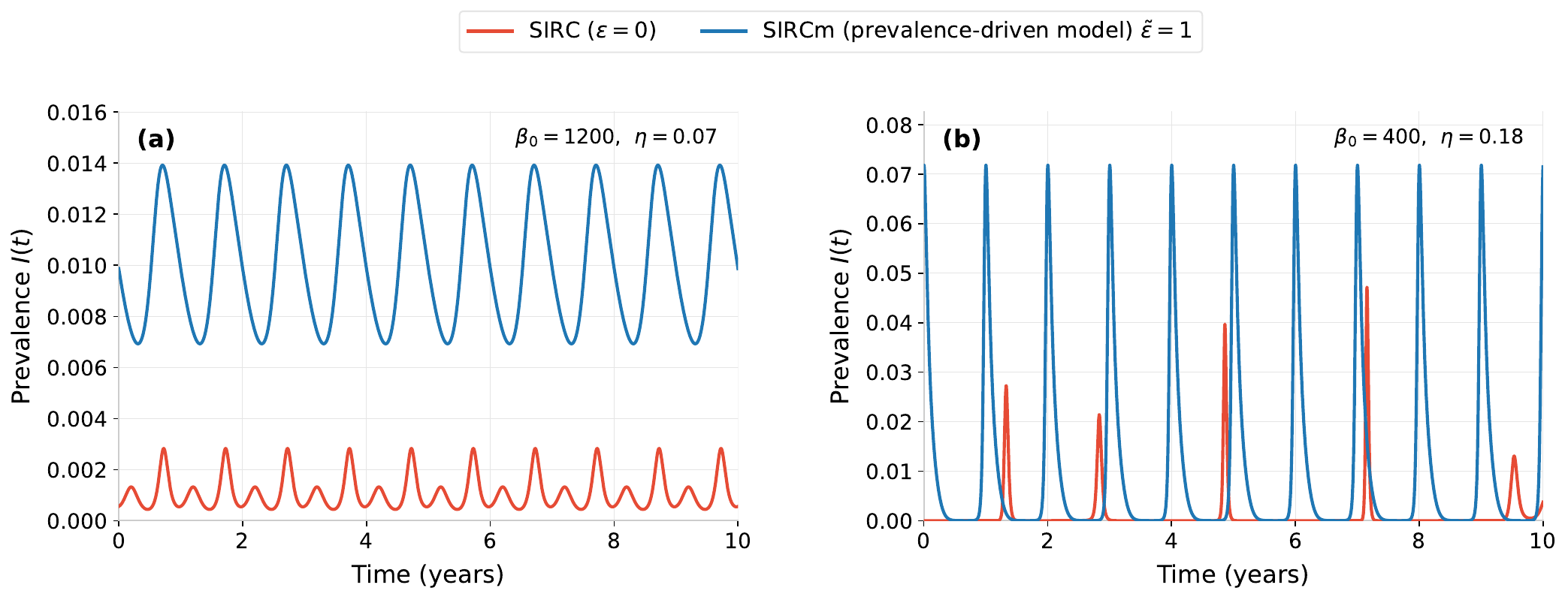}
    \caption{(see Appendix \ref{app:seasonal}) As Figure~\ref{fig:seasonalforcing-example}, at $\tilde\varepsilon = 1$. In the tropical regime (a), SIRCm now oscillates as a single-peaked annual cycle at substantially higher prevalence. In the temperate regime (b), SIRCm settles into a regular annual cycle, in contrast with the chaotic SIRC baseline}    
    \label{fig19}
\end{figure}

\begin{figure}
    \centering
    \includegraphics[width=1\textwidth]{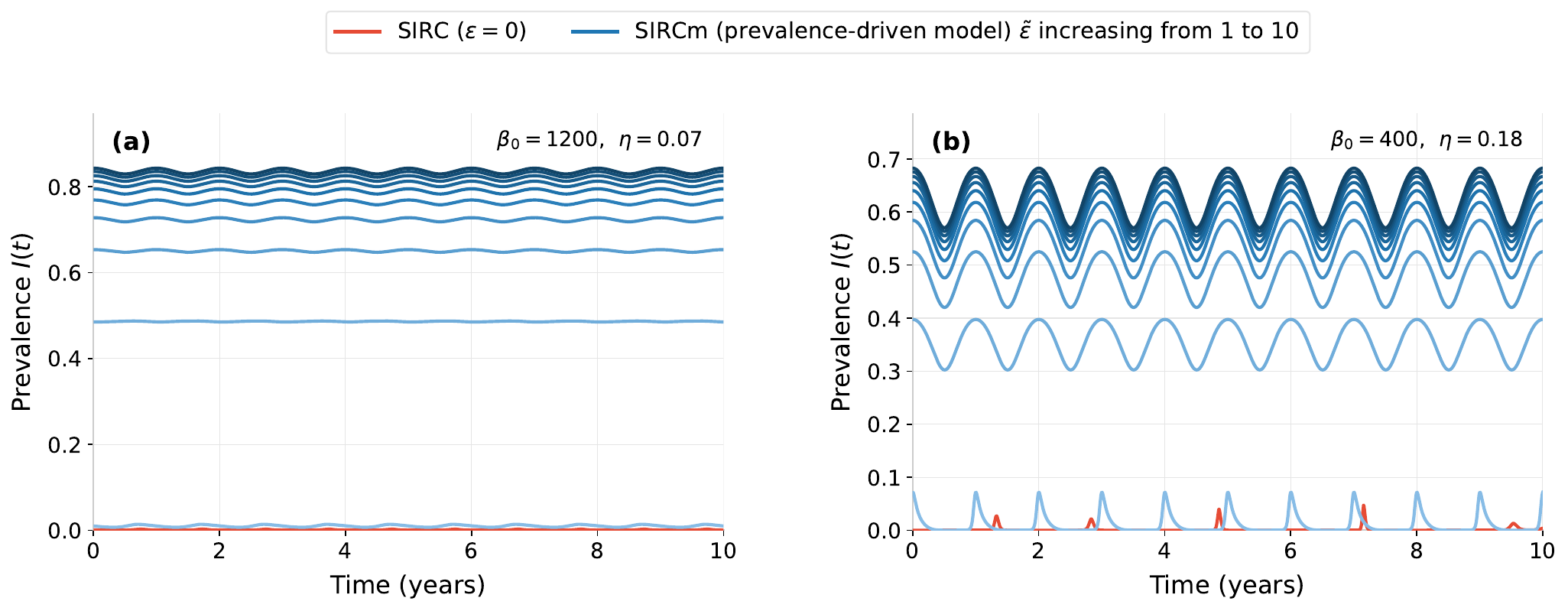}
    \caption{(see Appendix \ref{app:seasonal}) Family of prevalence-driven SIRCm time series for $\tilde{\varepsilon}$ ranging from $1$ to $10$ (increasingly dark blue), compared against the SIRC baseline ($\tilde{\varepsilon}=0$, negligible at this scale). In both regimes, SIRCm remains on a smooth, single-peaked annual cycle throughout, and its mean prevalence rises toward the theoretical ceiling $I^*_\infty = 1-\mathcal{R}_0^{-1}$ as $\tilde{\varepsilon}$ increases. In the tropical regime (a), the oscillation amplitude collapses, leaving an almost constant high prevalence. In the temperate regime (b), a visible annual oscillation persists, with the curves instead compressing toward the upper end of the range as $\tilde{\varepsilon}$ grows}    
    \label{fig20}
\end{figure}

\end{appendices}

\newpage

\bibliographystyle{sn-mathphys-ay}
\bibliography{references}

\end{document}